\documentclass[10pt]{amsart}
\usepackage{graphicx, color, enumerate}
\usepackage{amscd}
\usepackage{amsmath}
\usepackage{amsfonts}
\usepackage{amssymb}
\usepackage{mathrsfs}

\newtheorem{theorem}{Theorem}[section]
\newtheorem{lemma}[theorem]{Lemma}
\newtheorem{definition}[theorem]{Definition}
\newtheorem{proposition}[theorem]{Proposition}
\newtheorem{remark}[theorem]{Remark}

\newtheorem{claim}{Claim}

\numberwithin{equation}{section}
\newcommand{\RR}{\mathbb R}

\renewcommand{\le}{\leqslant}
\renewcommand{\leq}{\leqslant}
\renewcommand{\ge}{\geqslant}
\renewcommand{\geq}{\geqslant}
\DeclareMathOperator*{\limm}{lim}
\DeclareMathOperator*{\minn}{min}
\DeclareMathOperator*{\supp}{supp}

\begin{document}
\hfill\today\bigskip

\title[A critical fractional equation with concave-convex nonlinearities]{A critical fractional equation with concave-convex power nonlinearities}

\thanks{B.B., E.C., and F.S. were partially supported by Research Project of MICINN-Spain (Ref. MTM2010-18128). E.C. was also partially supported by Research Project of MICINN-Spain (Ref. MTM2009-10878). R.S. was supported by
the MIUR National Research Project {\it Variational and Topological Methods
in the Study of Nonlinear Phenomena}, by the GNAMPA Project {\it Nonlocal Problems of Fractional Laplacian Type} and by the
ERC grant $\epsilon$ ({\it Elliptic Pde's and Symmetry of Interfaces and Layers
for Odd Nonlinearities}).}

\author[B. Barrios]{B. Barrios}
\address{Departamento de Matem\'aticas,
          Universidad Aut\'onoma de Madrid,
          28049 Madrid, Spain, And
         Instituto de Ciencias Matem\'aticas, (ICMAT-UAM-UC3M), C/Nicol\'as Cabrera, 15, 28049 Madrid, Spain}
\email{\tt bego.barrios@uam.es}

\author[E. Colorado]{E. Colorado}
\address{Departamento de Matem\'aticas,
         Universidad Carlos III de Madrid, 28911
         Legan\'es (Madrid), Spain, And
         Instituto de Ciencias Matem\'aticas, (ICMAT-UAM-UC3M), C/Nicol\'as Cabrera, 15, 28049 Madrid, Spain}
\email{\tt ecolorad@math.uc3m.es}

\author[R. Servadei]{R. Servadei}
\address{Dipartimento di Matematica e Informatica,
          Universit\`a della Calabria,
          Ponte Pietro Bucci 31 B, 87036 Arcavacata di Rende (Cosenza), Italy}
\email{\tt servadei@mat.unical.it}

\author[F. Soria]{F. Soria}
\address{Departamento de Matem\'aticas,
          Universidad Aut\'onoma de Madrid,
          28049 Madrid, Spain, And
         Instituto de Ciencias Matem\'aticas, (ICMAT-UAM-UC3M), C/Nicol\'as Cabrera, 15, 28049 Madrid, Spain}
\email{\tt fernando.soria@uam.es}



\begin{abstract}
In this work we study the following fractional critical problem
$$
(P_{\lambda})=\left\{
\begin{array}{ll}
(-\Delta)^s u=\lambda u^{q} + u^{2^*_{s}-1}, \quad u{>}0 &  \mbox{in } \Omega\\
u=0 & \mbox{in } \RR^n\setminus \Omega\,,
\end{array}\right.
$$
where  $\Omega\subset \mathbb{R}^n$ is a regular bounded domain, $\lambda>0$, $0<s<1$ and $n>2s$. Here
$(-\Delta )^s$ denotes the fractional Laplace operator defined, up to a normalization factor, by
$$
-(-\Delta)^s u(x)={\rm P. V.}
\int_{\RR^n}\frac{u(x+y)+u(x-y)-2u(x)}{|y|^{n+2s}}\,dy,
\quad x\in \RR^n.
$$
Our main results show the existence and multiplicity of solutions to problem $(P_\lambda)$ for different values of $\lambda$. The dependency on this parameter changes according to whether we consider the concave power case ($0<q<1$) or the convex power case ($1<q<2^*_s-1$). These two cases will be treated separately.
\end{abstract}

\maketitle

\tableofcontents

\section{Introduction}\label{sec:introduzione}
In recent years, considerable attention has been given to nonlocal diffusion problems, in particular to the ones driven by the fractional Laplace operator. One of the reasons for this comes from the fact that this operator naturally arises in several physical phenomena like flames propagation and chemical reactions of liquids, in population dynamics and geophysical fluid dynamics, or  in mathematical  finance (American options). It also provides a simple model to describe certain jump L\'{e}vy processes in probability theory.  In all these cases, the nonlocal effect is modeled by the singularity at infinity.  For more details and  applications, see \cite{apli1, apli2, apli3, apli4, apli5, apli6} and the references therein.

\

In this paper we focus our attention on critical nonlocal fractional problems. To be more precise, we consider the following critical problem with convex-concave nonlinearities
$$
(P_{\lambda})=\left\{
\begin{array}{ll}
(-\Delta)^s u=\lambda u^q+u^{2^*_{s}-1} & {\mbox{ in }} \Omega,\\
u>0 & {\mbox{ in }} \Omega,\\
u=0 & {\mbox{ in }} \RR^n\setminus \Omega\,,
\end{array} \right.
$$
where $\Omega\subset \RR^n$ is a regular bounded domain, $\lambda>0$\,, $n>2s$, $0<q< 2^*_{s}-1$ and
\begin{equation}\label{2star}
2^*_{s}=\frac{2n}{n-2s},
\end{equation}
is the fractional critical Sobolev exponent. Here $(-\Delta)^s$ is the fractional Laplace operator defined, up to a normalization factor, by the Riesz potential as
\begin{equation} \label{2}
-(-\Delta)^s u(x):={\rm P. V.}
\int_{\RR^n}\frac{u(x+y)+u(x-y)-2u(x)}{|y|^{n+2s}}\,dy\,,
\,\,\,\,\, x\in \RR^n\,,
\end{equation}
where $s\in (0,1)$ is a fixed parameter
(see \cite[Chapter 5]{Stein} or \cite{valpal, sil} for further details).

One can also define a fractional power of the Laplacian using spectral decomposition. The same problem considered here but for this spectral fractional Laplacian has been treated in \cite{bcps}. Some related problems involving this operator have been studied in \cite{colorado1, cabre-tan, col-dep-san, tan}. As in  \cite{bcps} the purpose of this paper is to study the existence of weak solutions for $(P_{\lambda})$. Previous works related to the operator defined in \eqref{2}, or by a more general kernel, can be found in \cite{FallWeth,molicaBC, pp, quim-ros, sY, sBNRES, servadeivaldinociBN, servadeivaldinociBNLOW, servadeivaldinociCFP}.

Problems similar to $(P_\lambda)$ have been also studied in the local setting with different  elliptic operators. As far as we know, the first example in this direction was given in \cite{G-Peral} for the $p$-Laplacian operator. Other results, this time for the Laplacian (or essentially the classical Laplacian) operator can be found in \cite{AbCP,ABC,BEP,colorado}. More generally, the case of fully nonlinear operators has been studied in \cite{cha-col-per}.

It is worth noting here that the problem $(P_{\lambda})$, with $\lambda =0$, has no solution whenever $\Omega$ is a star-shaped domain. This has been proved in \cite{fsvnew,quim-ros2} using a Pohozaev identity for the operator $(-\Delta)^s$. This fact motivates the  perturbation term $\lambda u^q$, $\lambda>0$, in our work.

\

We now summarize the main results of the paper. First, in Section~\ref{sec:q<1} we look at the problem~$(P_{\lambda})$ in the concave case~$q<1$ and  prove the following.
\begin{theorem}\label{lapfra0q}
Assume $0<q<1$, $0<s<1$, and $n> 2s$.
Then, there exists $0<\Lambda <\infty$ such that problem~$(P_{\lambda})$
\begin{enumerate}
\item has no solution for $\lambda>\Lambda$;
\item has a minimal solution for any $0<\lambda <\Lambda$; moreover, the family of minimal solutions is
increasing with respect to $\lambda$;
\item if $\lambda=\Lambda$ there exists at least one solution;
\item for $0<\lambda<\Lambda$ there are at least two solutions.
\end{enumerate}
\end{theorem}

The convex case is treated in Section~\ref{sec:q>1}. The existence result for problem~$(P_{\lambda})$ is given by:
\begin{theorem}\label{lapfra0}
Assume $1<q<2^*_{s}-1$, $0<s<1$, and  $n>2s$.
Then, problem~$(P_{\lambda})$ admits at least one  solution provided that either
\begin{itemize}
\item $n>\frac{2s(q+3)}{q+1}$ and $\lambda>0$, or
\item $n\leq \frac{2s(q+3)}{q+1}$ and $\lambda$ is sufficiently large.
\end{itemize}
\end{theorem}

Theorem~\ref{lapfra0q} corresponds to the nonlocal version of the main result of \cite{ABC}, while Theorem~\ref{lapfra0} may be seen as the nonlocal counterpart of the results obtained for the standard Laplace operator in \cite[Subsection 2.3, 2.4 and 2.5]{bn}, (see also \cite[Theorem 3.2 and 3.3]{G-Peral} for the case of the $p$-Laplacian operator). Note, in particular, that when $s=1$ one has $2s(q+3)/(q+1)=2(q+3)/(q+1)<4$, due to the choice of $q>1$.

We will denote by $H^s(\RR^n)$ the usual fractional
Sobolev space endowed with the so-called \emph{Gagliardo norm}
\begin{equation}\label{gagliardonorm}
\|g\|_{H^s(\RR^n)}=\|g\|_{L^2(\RR^n)}+
\left(\int_{\RR^n\times \RR^n}\frac{\,\,\,|g(x)-g(y)|^2}{|x-y|^{n+2s}}\,dx\,dy\right)^{1/2}\,,
\end{equation}
while $X_0^s(\Omega)$ is the function space defined as
\begin{equation}\label{X0}
X_0^s(\Omega)=\big\{u\in H^s(\RR^n) : u=0\,\, \mbox{a.e. in}\,\, \RR^n\setminus \Omega\}\,.
\end{equation}
We refer to \cite{svmountain, servadeivaldinociBN} for a general definition of $X_0^s(\Omega)$ and its properties and to \cite{adams, valpal, mazya} for an account of the properties of $H^s(\RR^n)$.

In $X_0^s(\Omega)$ we can consider the following norm
$$
\|v\|_{X_0^s(\Omega)}=\left(\int_{\RR^n\times \RR^n}\frac{\,\,\,|v(x)-v(y)|^2}{|x-y|^{n+2s}}\,dx\,dy\right)^{1/2}.
$$
We also recall that $\left(X_0^s(\Omega), \|\cdot\|_{X_0^s(\Omega)}\right)$ is a Hilbert space, with scalar product
\begin{equation}\label{scalarproduct}
\langle u,v\rangle_{X_0^s(\Omega)}=\int_{\RR^n\times \RR^n}
\frac{\big( u(x)-u(y)\big) \big( v(x)-v(y)\big)}{|x-y|^{n+2s}}\,dx\,dy\,.
\end{equation}
See for instance \cite[Lemma~7]{svmountain}.

Observe that by \cite[Proposition 3.6]{valpal} we have the following identity
\begin{equation}\label{motorista}
\| u\|_{X_0^s(\Omega)}=\|(-\Delta)^{s/2}u\|_{L^2(\mathbb{R}^n)}.
\end{equation}
This leads us to establish as a definition that the solutions to our problem in this variational framework are those functions satisfying the relationship \eqref{def} below.

In our context, the  Sobolev constant is given by
\begin{equation}\label{Ss}
{\displaystyle S(n,s):=\inf_{v\in H^s(\RR^n)\setminus\{0\}}Q_{n,s}(v)>0,}
\end{equation}
where
$$
Q_{n,s}(v):=\frac{{\displaystyle\int_{\RR^n\times \RR^n}\frac{|v(x)-v(y)|^2}{|x-y|^{n+2s}}\,dx\,dy}}{{\displaystyle\left(\int_{\mathbb{R}^{n}} |v(x)|^{2^*_{s}}dx\right)^{2/2^*_{s}}}}\,, \quad v\in H^s(\mathbb{R}^n),
$$
is the associated Rayleigh quotient.
The constant $S(n,s)$ is well defined, as can be seen in \cite[Theorem 7.58]{adams}.

\

\subsection{Variational formulation of the problem}

\

Let us start describing the notion of solution in this context. In order to present the weak formulation of $(P_{\lambda})$ and taking into account that  we are looking for positive solutions, we will consider the following Dirichlet problem
$$
(P_{\lambda}^{+})=\left\{
\begin{array}{ll}
(-\Delta)^s u=\lambda (u_+)^q+(u_+)^{2^*_{s}-1} & {\mbox{ in }} \Omega,\\
u=0 & {\mbox{ in }} \RR^n\setminus \Omega\,,
\end{array} \right.
$$
where $u_+:=\max\{u,0\}$ denotes  the positive part of $u$. With this at hand, we  can now give the following.

\begin{definition}\label{problemalapfrac0}
We say that $u\in X_{0}^{s}(\Omega)$ is a weak solution of $(P_{\lambda}^{+})$ if for every $\varphi\in X_{0}^{s}(\Omega)$, one has
\begin{equation} \label{def}
\displaystyle \int_{\RR^n\times \RR^n}{\frac{(u(x)-u(y))(\varphi(x)-\varphi(y))}{|x-y|^{n+2s}}\,dx\,dy} =\lambda
\int_\Omega {(u_+)^q \varphi\,dx}+\int_\Omega {(u_+)^{2^*_{s}-1}\varphi\, dx.}
\end{equation}
\end{definition}

\

In the sequel we will omit the term {\it weak} when referring
to solutions that satisfy the conditions of Definition \ref{problemalapfrac0}.
The crucial observation here is that, by the Maximum Principle \cite[Proposition 2.2.8]{sil}, if $u$ is a solution of $(P_{\lambda}^{+})$ then $u$ is strictly positive in $\Omega$ and, therefore, it is also a solution of $(P_{\lambda})$.

\

To find solutions of ${(P_{\lambda}^+)}$, we will use a variational approach. Hence, we will associate a suitable functional to our problem. More precisely, the Euler--Lagrange functional related to problem~${(P_{\lambda}^+)}$ is given by
$\mathcal J_{s,\,\lambda}:X_0^s(\Omega)\to \RR$ defined as follows
$$\mathcal J_{s,\,\lambda}(u)=\frac 1 2 \int_{\RR^n\times \RR^n}\frac{|u(x)-u(y)|^2}{|x-y|^{n+2s}}\,dx\,dy-\frac{\lambda}{q+1} \int_\Omega (u_+)^{q+1}\, dx-\frac{1}{2^*_{s}}\int_\Omega (u_+)^{2^*_{s}}\,dx\,.$$
Note that $\mathcal{J}_{s,\,\lambda}$ is $\mathcal{C}^{1}$ and that its critical points correspond to solutions of ${(P_{\lambda}^+)}$.

\

In both cases, $q<1$ and $q>1$, we will use  the Mountain Pass Theorem (MPT) by Ambrosetti and Rabinowitz (see \cite{ar}). In order to do that, we will show that $\mathcal J_{s,\,\lambda}$ satisfies a compactness property and has suitable geometrical features. The fact that the functional has the suitable geometry is easy to check.  Observe that the embedding $X_0^s(\Omega)\hookrightarrow L^{2^*_{s}}(\RR^n)$ is not compact (see \cite{adams} ). This is even true when the nonlocal operator has  a more general kernel (see  \cite[Lemma 9-b)]{servadeivaldinociBN}). Hence, the difficulty to apply MPT lies on proving a local Palais--Smale (PS for short) condition at level $c\in \RR$ ((PS)$_c$).
Moreover, since the PS condition does not hold globally, we have to prove that the Mountain Pass critical level of $\mathcal J_{s,\,\lambda}$ lies below the threshold of application of the (PS)$_c$ condition.

In the concave setting, $q<1$, the idea is to prove the existence of at least two positive solutions for an admissible small range of $\lambda$. For that we are using a contradiction argument, inspired by \cite{ABC}. The proof is divided into several steps: we first show that we have a solution that is a local minimum for the functional $\mathcal J_{s,\,\lambda}$. In the next step, in order to find a second solution,  we suppose that this local minimum is the only critical point of the functional, and then we prove a local (PS)$_c$ condition for $c$ under a critical level related with the best fractional critical Sobolev constant given in \eqref{Ss}. Also we find a path under this critical level localizing the Sobolev minimizers at the possible concentration on Dirac Deltas. These Deltas are obtained by the concentration-compactness result in \cite[Theorem~1.5]{pp} inspired in the classical result by P.L. Lions in \cite{Lions1,Lions2}. Applying the MPT given in \cite{ar} and its refined version given in \cite{GhP}, we will reach a contradiction.

In the convex case $q>1$ we also apply the MPT  to obtain the existence of at least one solution for ${(P_{\lambda}^+)}$ for suitable values of $\lambda$ depending on the dimension $n$.  As before, we prove a local (PS)$_c$ condition in a appropriate range related with the constant $S(n,s)$ defined on \eqref{Ss}.
The strategy to obtain a  solution follows the ideas given in \cite{bn} (see also \cite{struwe, willem}) adapted to our nonlocal functional framework.

The linear case $q=1$, when the right hand side of the equation is equal to $\lambda u+|u|^{2^*_s-2}u$, was treated in \cite{sY, sBNRES, servadeivaldinociBN, servadeivaldinociBNLOW, servadeivaldinociCFP}. {In these works} the authors studied also nonlinearities more general than those given by the power critical function {as well as} the existence of solutions not necessarily positive.


\section{The critical and concave case $0<q<1$}\label{sec:q<1}
This section is devoted to the study of problem~$(P_{\lambda})$ in the case of the exponent $0<q<1$. We point out that the result of Theorem \ref{lapfra0q} in the subcritical case  could be obtained by the arguments given in this paper.  However, in this subcritical case the PS condition is easier to prove  -it is indeed satisfied for any energy level- and the separation of solutions, presented  in Lemma \ref{separacion} below, is not needed. This approach has been carried out in \cite{bmp} where the authors obtain the equivalent to Theorem \ref{lapfra0q} for a related problem using a technique developed in \cite{Alama}.

We begin with the following result that uses, in its proof,  a standard comparison method as well as some ideas given in \cite[Lemma 3.1 and Lemma 3.4]{ABC}.

\begin{lemma}\label{lem:>}
Let $0<q<1$ and let $\Lambda$ be defined by
\begin{equation}\label{Lambda}
\Lambda :=\sup\Big\{ \lambda >0\ :\ \mbox{problem~$(P_{\lambda})$ has solution}\Big\}.
\end{equation}

Then, $0<\Lambda <\infty$ and the critical concave problem~$(P_{\lambda})$ has at least one solution for every $0<\lambda\leq\Lambda$. Moreover, for $0<\lambda<\Lambda$ we get a family of minimal solutions increasing with respect to $\lambda$.
\end{lemma}

By Lemma~\ref{lem:>} we easily deduce statements~$(1)-(3)$ of Theorem~\ref{lapfra0q}. Hence, in the sequel we focus on proving statement~$(4)$ of that theorem, that is on the existence of a second solution for~$(P_{\lambda})$.

First we prove a regularity result which will be useful in certain parts of this section:
\begin{proposition}\label{prop:acotacion2}
Let $u$ be a positive solution to the problem
\begin{equation*}
\left\{
    \begin{array}{rcl@{\quad}l}
    (-\Delta)^{s}u & = & f(x, u)&\mbox{ in }\  \Omega,\\
    u & = & 0&\mbox{ in }\  \mathbb{R}^n\setminus \Omega,
    \end{array}
    \right.
    \end{equation*}
    and assume that $|f(x,t)|\le C(1+|t|^p)$, for some $1\leq p\leq 2^*_s-1$ and $C>0$.
    Then $u\in L^\infty(\Omega)$.

\end{proposition}
\begin{proof}

The proof uses standard techniques for the fractional Laplacian, in particular the following inequality: if $\varphi$ is a convex and differentiable function, then
$$
(-\Delta)^{s}\varphi(u)\le \varphi'(u) \, (-\Delta)^{s}u.
$$
Let us define, for $\beta \ge 1$ and $T>0$ large,
$$
\varphi(t)=\varphi_{T,\beta}(t)=\left\{
\begin{array}{ll}
0, \quad&{\rm if} \quad t\leq0\\[2mm]
t^\beta, \quad&{\rm if} \quad 0<t<T\\[2mm]
\beta T^{\beta-1} (t-T)+T^\beta, \quad&{\rm if} \quad t\ge T.
\end{array}\right.
$$
 Observe that $\varphi(u) \in X_0^s(\Omega)$ since $\varphi$ is Lipschitz with constant $K=\beta T^{\beta-1} $ and, therefore,
 $$
\begin{array}{ll}
\| \varphi(u)\|_{X_0^s(\Omega)} &=\displaystyle
\left( \int_{\mathbb R^{n}\times\mathbb R^{n}}\frac{\,\,\,|\varphi(u(x))-\varphi(u(y))|^2}{|x-y|^{n+2s}}\,dx\,dy\right)^{1/2}\\[2mm]
&\le \displaystyle  \left( \int_{\mathbb R^{n}\times\mathbb R^{n}}\frac{\,\,\,K^2|u(x)-u(y)|^2}{|x-y|^{n+2s}}\,dx\,dy\right)^{1/2}=K
\| u\|_{X_0^s(\Omega)}.
\end{array}
$$
By \eqref{motorista} and the Sobolev embedding theorem given in \cite[Theorem 7.58]{adams}, we have
\begin{equation}\label{a}
\int_{\Omega} \varphi(u)(-\Delta)^{s}\varphi(u)=\| \varphi(u)\|^2_{X_0^s(\Omega)}\geq S(n,s)
 \| \varphi(u) \|^2_{L^{2^*_s}(\Omega)},
\end{equation}
where $S(n,s)$ is defined in \eqref{Ss}.
On the other hand, since $\varphi$ is convex, and $\varphi(u)\varphi'(u)\in X_{0}^{s}(\Omega)$,
$$ \int_\Omega \varphi(u)(-\Delta)^{s}\varphi(u)\leq \int_\Omega \varphi(u)\varphi'(u) \, (-\Delta)^{s}u \le C \int_\Omega  \varphi(u)\varphi'(u)\, \left(1+u^{2^*_s-1}\right). $$
From \eqref{a} and the previous inequality we get the following basic estimate:
\begin{equation}  \label{one}
 \| \varphi(u) \|^2_{L^{2^*_s}(\Omega)} \le C \int_\Omega  \varphi(u)\varphi'(u)\, \left(1+u^{2^*_s-1}\right).
\end{equation}

\

\noindent Since $u\,\varphi'(u)\le \beta \, \varphi(u)$ and $\varphi'(u)\le \beta\, (1+\varphi(u))$, the above estimate (\ref{one}) becomes
\begin{equation} \label{two}
\left(\int_\Omega \left(\varphi(u)\right)^{2^*_s}\right)^{2/{2^*_s}} \le C\, \beta \,
\left(1+\int_\Omega \left(\varphi(u)\right)^{2}+\int_\Omega \left(\varphi(u)\right)^{2}u^{2^*_s-2} \right).
\end{equation}
It is important to point out here that since $\varphi(u)$ grows linearly, both sides of (\ref{two}) are finite.

\

\noindent {\bf Claim}: Let $\beta_1$ be such that $2\beta_1=2^*_s$. Then $u\in L^{\beta_1 \, 2^*_s}$.

\

\noindent To see this, we take R large to be determined later. Then, H\"older's inequality with $p=\beta_1=2^*_s/2$ and $p'=2^*_s/(2^*_s-2)$ gives
$$
\begin{array}{ll}
\displaystyle \int_\Omega \left(\varphi(u)\right)^{2}u^{2^*_s-2}
&= \displaystyle \int_{\{u\le R\}} \left(\varphi(u)\right)^{2}u^{2^*_s-2}+ \int_{\{u>R\}} \left(\varphi(u)\right)^{2}u^{2^*_s-2}\\[5mm]
& \le \displaystyle \int_{\{u\le R\}} \left(\varphi(u)\right)^{2}R^{2^*_s-2}\\[5mm]
&+ \displaystyle \left(\int_\Omega \left(\varphi(u)\right)^{2^*_s}\right)^{2/{2^*_s}} \left(\int_{\{u>R\}} u^{2^*_s}\right)^{(2^*_s-2)/2^*_s}.
\end{array}
$$
By the Monotone Convergence Theorem, we may take $R$ so that
$$
\left(\int_{\{u>R\}} u^{2^*_s}\right)^{(2^*_s-2)/2^*_s}\le \frac 1{2 \, C\, \beta_1}.
$$
In this way, the second term above is absorbed by the left hand side of (\ref{two}) to get
\begin{equation} \label{three}
\left(\int_\Omega \left(\varphi(u)\right)^{2^*_s}\right)^{2/{2^*_s}} \le 2\, C\, \beta_1 \,
\left(1+\int_\Omega \left(\varphi(u)\right)^{2}+\int_{\{u\le R\}} \left(\varphi(u)\right)^{2}R^{2^*_s-2} \right).
\end{equation}
Using that $\varphi_{T,\beta_1}(u)\leq u^{\beta_1}$  in the right hand side of (\ref{three}) and then letting $T\longrightarrow \infty$ in the left hand side, since $2\beta_1=2^*_s$, we obtain
\begin{equation*}
\left(\int_\Omega u^{2^*_s\beta_1}\right)^{2/{2^*_s}} \le 2\, C\, \beta_1 \,
\left(1+\int_\Omega u^{2^*_s}+R^{2^*_s-2} \int_\Omega u^{2^*_s}\right)<\infty.
\end{equation*}
This proves the claim.

\

We now go back to inequality (\ref{two}) and we use as before  that $\varphi_{T,\beta}(u)\leq u^{\beta}$ in the rigth hand side and then we take $T\longrightarrow \infty$ in the left hand side. Then,
\begin{equation*}
\left(\int_\Omega u^{2^*_s\beta}\right)^{2/{2^*_s}} \le C\, \beta \,
\left(1+\int_\Omega u^{2\beta}+\int_\Omega u^{2\beta+2^*_s-2} \right).
\end{equation*}
Since $\displaystyle \int_\Omega u^{2\beta}\le |\Omega|\,+ \int_\Omega u^{2\beta+2^*_s-2}$,
we get  the following recurrence formula
$$
\left(\int_\Omega u^{2^*_s\beta}\right)^{2/{2^*_s}}  \le
2\,  C\, \beta \, (1+|\Omega|) \left(1+\int_\Omega u^{2\beta+2^*_s-2} \right).
$$
Therefore,
\begin{equation}\label{c}
\left(1+\int_\Omega u^{2^*_s\beta}\right)^{\frac{1}{2^*_s(\beta-1)}}  \le C_{\beta}^{\frac{1}{2(\beta-1)}} \left(1+\int_\Omega u^{2\beta+2^*_s-2} \right)^{\frac{1}{2(\beta-1)}},
\end{equation}
where $C_{\beta}=4\,  C\, \beta \, (1+|\Omega|)$.

\noindent For $m\ge 1$ we define $\beta_{m+1}$ inductively so that $2\beta_{m+1}+ 2^*_s-2=2^*_s\beta_m$, that is
$$\beta_{m+1}-1=\frac {2^*_s} 2(\beta_m-1)=\left(\frac {2^*_s} 2\right)^m(\beta_1-1).$$
Hence, from \eqref{c} it follows that
$$
\left(1+\int_\Omega u^{2^*_s\beta_{m+1}}\right)^{\frac{1}{2^*_s(\beta_{m+1}-1)}}  \le C_{\beta_{m+1}}^{\frac{1}{2(\beta_{m+1}-1)}} \left(1+\int_\Omega u^{2^*_s\beta_m} \right)^{\frac{1}{2^*_s(\beta_m-1)}},
$$
with $C_{m+1}:=C_{\beta_{m+1}}=4\,  C\, \beta_{m+1} \, (1+|\Omega|)$.
Then, defining for $m\geq 1$
$$A_m:=\left(1+\int_\Omega u^{2^*_s\beta_m}\right)^{\frac{1}{2^*_s(\beta_m-1)}},$$
by the Claim proved before, and using a limiting argument, we conclude that there exists $C_0>0$, independent of $m>1$, such that
$$A_{m+1}\leq \prod_{k=2}^{m+1}{C_{k}^{\frac{1}{2(\beta_k-1)}}}A_1\leq C_0A_1.$$
This implies that $\|u\|_{L^\infty(\Omega)}\leq  C_0A_1.$
 \end{proof}


Coming back to the proof Theorem~\ref{lapfra0q}, as we said in the Introduction, to find the existence of the second solution, we first show that the minimal solution ${u_{\lambda}>0}$ given by Lemma~\ref{lem:>} is a local minimum for the functional $\mathcal J_{s,\,\lambda}$. For that, following the ideas given in \cite{colorado} we establish a separation lemma in the topology of the class
\begin{equation}\label{space}
\mathcal{C}_{s}(\Omega):=\left\{w\in \mathcal{C}^{0}(\overline{\Omega}): \, \|w\|_{\mathcal{C}_{s}(\Omega)}:=\left\|\frac{w}{\delta^{s}}\right\|_{L^{\infty}(\Omega)}<\infty\right\},
\end{equation}
where $\delta(x)=\rm{dist} (x,\partial\Omega)$.
Then we have the following.
\begin{lemma}\label{separacion}
Assume $0<\lambda_1<\lambda_0<\lambda_2<\Lambda$. Let $u_{\lambda_1}$,
$u_{\lambda_0}$ and $u_{\lambda_2}$ be the corresponding minimal
solutions to $(P_{\lambda})$, for $\lambda=\lambda_1,\,\lambda_0$ and
$\lambda_2$ respectively. If
$$Z=\{u\in \mathcal{C}_s(\Omega)|
\,u_{\lambda_{1}}\leq u \leq u_{\lambda_{2}} \},$$ then there exists
$\varepsilon
>0$ such that
$$
\{u_{\lambda_0}\} + \varepsilon B_{1} \subset Z,
$$
with $B_{1}=\{w\in \mathcal{C}^{0}(\overline{\Omega}): \,\|\frac{w}{\delta^{s}}\|_{L^{\infty}(\Omega)}<1\}.$
\end{lemma}
\begin{proof}
Let $u$ be an arbitrary solution of $(P_{\lambda})$ for $0<\lambda<\Lambda$. Then, by Hopf's Lemma (see \cite[Proposition~2.7]{crs} and \cite[Lemma 3.2]{quim-ros}) there exists
a positive constant $c$ such that
\begin{equation}\label{distancia2}
u(x)\geq c\delta(x)^{s},\,x\in\Omega.
\end{equation}
On the other hand by \cite[Proposition 1.1]{quim-ros}  we get that there
exists a positive constant $C$ such that
\begin{equation}\label{distancia1}
u(x)\leq C\delta(x)^{s},\,x\in\Omega.
\end{equation}
Thus, by \eqref{distancia2} and \eqref{distancia1} we finish the proof.
\end{proof}

Using this previous result we now obtain  a local minimum of the functional
$\mathcal J_{s,\,\lambda}$ in the $\mathcal{C}_{s}(\Omega)$-topology. This is the first step in order to get a
local minimum in $X_0^s(\Omega)$. That is,
\begin{lemma}\label{minimum}
For all $\lambda \in (0,\Lambda)$ the minimal solution $u_{\lambda}$ is a local minimum of the functional $\mathcal J_{s,\,\lambda}$ in
the $\mathcal{C}_s$-topology.
\end{lemma}
\begin{proof}
The proof follows in a similar way as in \cite{ABC} (see also Lemma 3.3 of \cite{colorado}). In our case we have to consider the non local operator $(-\Delta)^s$ instead of $(-\Delta)$ and the space $\mathcal{C}_{s}(\Omega)$ instead of $\mathcal{C}_{0}^{1}(\Omega)$. We omit the details.
\end{proof}
To prove that we already have a minimum in  the space $X_0^s(\Omega)$ we show that the result obtained by Brezis and
Nirenberg in \cite{bn} is also valid in our context.
\begin{proposition}\label{alpha_vs_c1}
Let $z_{0}\in X_0^s(\Omega)$ be a local minimum of
$\mathcal J_{s,\,\lambda}$ in $\mathcal{C}_{s}(\Omega)$; by this we mean that there exists $r_{1}>0$ such
that
\begin{equation}\label{hhip}
\mathcal J_{s,\,\lambda}(z_{0})\leq \mathcal J_{s,\,\lambda}(z_{0}+z), \qquad \forall z\in
\mathcal{C}_{s}(\Omega)\mbox{ with }
\left\|z\right\|_{\mathcal{C}_{s}(\Omega)}\leq r_{1}.
\end{equation}
Then, $z_{0}$ is also a local minimum of $\mathcal J_{s,\,\lambda}$ in $X_0^s(\Omega)$, that is, there exists $r_{2}>0$ so that
$$
\mathcal J_{s,\,\lambda}(z_{0})\leq \mathcal J_{s,\,\lambda}(z_{0}+z), \qquad \forall z\in X_0^s(\Omega) \mbox{ with } \|z\|_{X_0^s(\Omega)}\leq r_{2}.
$$
\end{proposition}
\begin{proof}
We follow the ideas given in \cite[Theorem 5.1]{colorado}. Let $z_0$ be as in \eqref{hhip} and set, for $\varepsilon >0$,
$$B_{\varepsilon}(z_{0})=\left\{z\in X_0^s(\Omega):\, \,\|z-z_{0}\|_{X_0^s(\Omega)}\leq\varepsilon\right\}.$$
Now, we argue by contradiction and we suppose that for every  $\varepsilon>0$ we have
\begin{equation}\label{contradiction}
\min_{v\in B_{\varepsilon}(z_{0})}{\mathcal{J}_{s, \lambda}}(v)<\mathcal J_{s,\,\lambda}(z_{0})\,.
\end{equation}
We pick $ \displaystyle \,v_{\varepsilon}\in
B_{\varepsilon}(z_{0})$  such that 
$\displaystyle \min_{v\in B_{\varepsilon}(z_{0})}{\mathcal{J}_{s, \lambda}}(v)=\mathcal J_{s,\,\lambda}(v_{\varepsilon})$.
The existence of $v_{\varepsilon}$ comes from a standard  argument of weak lower semi-continuity.
We want to prove that
\begin{equation}\label{aim1}
v_{\varepsilon}\rightarrow z_{0} \quad \mbox{in}
\quad \mathcal{C}_{s}(\Omega) \quad \mbox{as} \quad \varepsilon\searrow 0\,,
\end{equation}
because this would imply  that there are $z\in {\mathcal{C}_s(\Omega)}$, arbitrarily close to $z_0$ in the metric of ${\mathcal{C}_s(\Omega)}$ (in fact,  $z=v_\varepsilon$ for some $\varepsilon$), such  that 
$$
\mathcal J_{s,\,\lambda}{(z)}< \mathcal J_{s,\,\lambda}(z_{0}).
 $$
This contradicts our hypothesis \eqref{hhip}.

\

Let $0<\varepsilon\ll1$. Note that the Euler--Lagrange
equation satisfied by $v_\varepsilon$ involves a Lagrange multiplier
$\xi_\varepsilon$ such that
\begin{equation}\label{eqeuler}
\langle
\mathcal{J}_{s,\lambda}'(v_{\varepsilon}),\varphi\rangle=\xi_{\varepsilon}\langle
v_{\varepsilon},\varphi\rangle_{X_0^s(\Omega)},\quad
\forall\,\varphi\in X_0^s(\Omega).
\end{equation}
As a consequence, since $v_{\varepsilon}$ is a minimum of $\mathcal{J}_{s,\lambda}$ in $B_{\varepsilon}(z_{0})$, we have
\begin{equation}\label{a-cero}
\xi_{\varepsilon}=\frac{\langle
\mathcal{J}_{s,\lambda}'(v_{\varepsilon}),
v_{\varepsilon}\rangle}{\|v_{\varepsilon}\|^{2}_{X_0^s(\Omega)}}\leq 0,\quad\mbox{ with } \xi_\varepsilon\to 0 \mbox{ when }
\varepsilon\searrow 0.
\end{equation}
By (\ref{eqeuler}) we easily get that
$v_{\varepsilon}$ satisfies
$$
\left\{
 \begin{array}{ll}
    (-\Delta)^sv_{\varepsilon}=\frac{1}{1-\xi_{\varepsilon}}f_{\lambda}(v_{\varepsilon})=:
    f_{\lambda}^{\varepsilon}(v_{\varepsilon})&\quad \mbox{in } \Omega\,, \\
    v_{\varepsilon}=0 & \quad \mbox{in } \mathbb{R}^{n}\setminus\Omega\,,
  \end{array}
  \right.
$$
where {$f_\lambda(t):=\lambda (t_+)^{q}+(t_+)^{2^*_{s}-1}\,.$}

Since $v_\varepsilon >0$ and
$$
\|v_{\varepsilon}\|_{X_0^s(\Omega)}\leq
C,$$
by Proposition~\ref{prop:acotacion2} there exists a constant $C_1>0$ independent of $\varepsilon$ such that
$\|v_{\varepsilon}\|_{L^{\infty}(\Omega) }\leq C_1$. Moreover, by \eqref{a-cero}, it follows that
$\|f_{\lambda}^{\varepsilon}(v_{\varepsilon})\|_{L^{\infty}(\Omega)}
\leq C$. Therefore, by
\cite[Proposition~1.1]{quim-ros} (see also \cite[Proposition~5]{Sv1}), we get that
$\|v_{\varepsilon}\|_{\mathcal{C}^{0,s}(\overline{\Omega})}\leq C_2$, for some $C_2$ independent of $\varepsilon$. Here $\mathcal{C}^{0,s}$ denotes the space of H\"{o}lder continuous functions with exponent $s$.

Thus, by the Ascoli-Arzel\'{a} Theorem there exists a subsequence, still
denoted by $v_{\varepsilon}$, such that $v_{\varepsilon}\rightarrow z_{0}$ uniformly  as $\varepsilon\searrow 0$.
Moreover, by \cite[Theorem~1.2]{quim-ros}, we obtain that for a suitable positive constant $C$
$$\Big\|\frac{v_{\varepsilon}-z_0}{\delta^s}\Big\|_{L^{\infty}(\Omega)}\leq C\sup_{\Omega}{\big|f_{\lambda}^{\varepsilon}(v_{\varepsilon})-f_{\lambda}(z_0)}\big|.$$
Since the latter tends to zero as $\varepsilon\searrow 0$\,,  \eqref{aim1} is proved.

\end{proof}

Lemma \ref{minimum} and Proposition \ref{alpha_vs_c1} provide us with the existence of a positive local minimum in $X^s_{0}(\Omega)$ of {$\mathcal{J}_{s,\lambda}$} that will be denoted by $u_{0}$. We now make a translation as in \cite{ABC} in order to
simplify the calculations.

For $0<\lambda<\Lambda$, we consider the functions
\begin{equation}
g_{\lambda}(x,t)=\left\{\begin{array}{ll} \lambda
(u_{0}+t)^{q}-\lambda u_{0}^{q} +(u_{0}+t)^{2^*_{s} -1}- u_{0}^{2^*_{s} -1}, &  \mbox{ if }t\ge 0,\\
0, &  \mbox{ if } t<0,
\end{array}
\right.
\label{g(x,s)}\end{equation}
and
\begin{equation}\label{G}
G_{\lambda}(x,\xi)=G_\lambda(\xi)=\int_{0}^{\xi}g_{\lambda}(x,t)\,dt.
\end{equation}
The associated energy functional $\widetilde{\mathcal J}_{s,\,\lambda}:X_0^s(\Omega) \to \RR$
is given by
\begin{equation}\label{II}
\widetilde{\mathcal J}_{s,\,\lambda}(u)= \frac{1}{2}\|u\|^{2}_{X_0^s(\Omega)}
- \int_{\Omega}G_{\lambda}(x,u)dx.
\end{equation}
Since $u\in X_0^s(\Omega)$, $\widetilde{\mathcal J}_{s,\,\lambda}$ is well defined. We define the translate problem
$$
   (\widetilde{P}_{\lambda})=\left\{\begin{array}{ll}
    (-\Delta)^su= g_{\lambda}(x,u)&\quad\mbox{in }
    \Omega\subset\mathbb{R}^n,\\
    u=0&\quad\mbox{on } \mathbb{R}^{n}\setminus\Omega.
  \end{array}\right.
$$

We know that if $\widetilde{u}
\not\equiv 0$ is a critical point of $\widetilde{\mathcal J}_{s,\,\lambda}$ then it is a solution of
$(\widetilde{P}_{\lambda})$ and, by the Maximum Principle (\cite[Proposition~2.2.8]{sil}), this implies that $\widetilde{u}>0$. Therefore
$u=u_{0}+\widetilde{u}>0$ will be a second solution of {$(P_{\lambda}^{+})$ and consequently a second one of $(P_{\lambda})$}. Hence, in order to prove statement~$(4)$ of Theorem~\ref{lapfra0q}, it is enough to study the existence of a non-trivial
critical point for $\widetilde{\mathcal J}_{s,\,\lambda}$.

First we have
\begin{lemma}\label{lemma:minimo}
$u=0$ is a local minimum of $\widetilde{\mathcal J}_{s,\,\lambda}$ in
$X_{0}^s(\Omega).$
\end{lemma}
\begin{proof}
The proof follows along the lines of \cite[Lemma 4.2]{ABC}, see also \cite[Lemma 3.4]{bcps}, so we omit the details.
\end{proof}
\subsection{The Palais--Smale condition for~$\widetilde{\mathcal J}_{s,\,\lambda}$}\label{sec:compactness1}
In this subsection assuming that we have a unique critical point, we prove that the functional $\widetilde{\mathcal J}_{s,\,\lambda}$ satisfies a local Palais--Smale condition (see Lemma~\ref{strongly_convergent}). The main tool for proving this fact is an extension of the concentration-compactness principle by Lions in \cite{Lions1, Lions2} for nonlocal fractional operators, given in \cite[Theorem~1.5]{pp}. We will also need some technical results related to the behavior of the fractional Laplacian of a product. We start with the following. 
\begin{lemma}\label{cero}
Let $\phi$ be a regular function that satisfies
\begin{equation}\label{cero1}
|\phi(x)|\leq\frac{\tilde C}{1+|x|^{n+s}},\, x\in\mathbb{R}^{n}
\end{equation}
and
\begin{equation}\label{cero2}
|\nabla\phi(x)|\leq\frac{\tilde C}{1+|x|^{n+s+1}},\, x\in\mathbb{R}^{n},
\end{equation}
for some $\tilde C>0$. Let $B:X_0^s(\Omega) \times X_0^s(\Omega) \to \RR$ be the bilinear form defined by
\begin{equation}\label{bilinear}
B(f,g)(x):=2 \, \int_{\mathbb{R}^{n}}{\frac{(f(x)-f(y))(g(x)-g(y))}{|x-y|^{n+s}}\, dy}\,.
\end{equation}

Then, for every $s\in(0,1)$, there exist positive constants $C_1$ and $C_2$ such that for $x\in \RR^n$ one has
$$|(-\Delta)^{s/2}\phi(x)|\leq\frac{C_1}{1+|x|^{n+s}},$$
and
$$|B(\phi,\phi)(x)|\leq\frac{C_2}{1+|x|^{n+s}}\,.$$
\end{lemma}
\begin{proof}
Let
$$I(x):=\int_{\mathbb{R}^{n}}{\frac{|\phi(x)-\phi(y)|}{|x-y|^{n+s}}\, dy}.$$
For any $x\in\mathbb{R}^{n}$, it is clear that
$$|(-\Delta)^{s/2}\phi(x)|\leq 2I(x).$$
Also, since $|\phi(x)|\leq \tilde{C}$, we have
$$|B(\phi,\phi)(x)|\leq 2\tilde{C} I(x).$$
Hence, it suffices to prove that
\begin{equation}\label{mars}
I(x)\leq \frac{C}{1+|x|^{n+s}},\quad\forall x\in\mathbb{R}^{n},
\end{equation}
for a suitable positive constante $C$.
\\Since $\phi$ is a regular function, for $|x|<1$ we obtain that,
\begin{eqnarray}\label{mars00}
I(x)&\leq&\|\nabla\phi\|_{L^{\infty}(\mathbb{R}^{n})}\int_{|y|<2}{\frac{dy}{|x-y|^{n+s-1}}}+C\int_{|y|\geq2}{\frac{dy}{|y|^{n+s}}}\nonumber\\
&\leq& C\leq \frac{C}{1+|x|^{n+s}}.
\end{eqnarray}
Let now $|x|\geq1$. Then
\begin{equation}\label{mars0}
I(x):=I_{A_1}(x)+I_{A_2}(x)+I_{A_3}(x),
\end{equation}
where
$$I_{A_i}(x):=\int_{A_{i}}{\frac{|\phi(x)-\phi(y)|}{|x-y|^{n+s}}\, dy},\, i=1,2,3,$$
with
$$A_1:=\{y:\, |x-y|\leq\frac{|x|}{2}\},\quad A_2:=\{y:\, |x-y|>\frac{|x|}{2},\, |y|\leq2|x|\}$$
and
$$A_3:=\{y:\, |x-y|>\frac{|x|}{2},\, |y|>2|x|\}.$$
Therefore, since for $|x|\geq1$ and $y\in A_1$, $|\phi(x)-\phi(y)|\leq|\nabla\phi(\xi)||x-y|$ with $\displaystyle\frac{|x|}{2}\leq|\xi|\leq\frac{3}{2}|x|$, by \eqref{cero2}, we obtain that
\begin{equation}\label{mars1}
I_{A_1}(x)\leq\frac{C}{|x|^{n+s+1}}\int_{A_1}{\frac{dy}{|x-y|^{n+s-1}}}\leq C|x|^{-(n+2s)}.
\end{equation}
Using now that, for any $x,\, y\in\mathbb{R}^{n}$ we have the inequality,
$$|\phi(x)|+|\phi(y)|\leq \frac{C}{1+\minn\{|x|^{n+s},|y|^{n+s}\}},$$
we get
\begin{equation}\label{mars2}
I_{A_2}(x)\leq \frac{C}{|x|^{n+s}}\int_{A_2}{\frac{dy}{(1+|y|^{n+s})}}\leq C|x|^{-(n+s)},
\end{equation}
and
\begin{equation}\label{mars3}
I_{A_3}(x)\leq\frac{C}{|x|^{n+s}}\int_{A_3}{\frac{dy}{|y|^{n+s}}}\leq C|x|^{-(n+2s)}.
\end{equation}
 Note that the last estimate follows from the fact that $(x,y)\in A_3$ implies $|x-y|\geq |y|/2$. Then, by \eqref{mars0}-\eqref{mars3}, we get that
\begin{equation}\label{mars000}
I(x)\leq C|x|^{-(n+s)}\leq\frac{C}{1+|x|^{n+s}},\quad|x|\geq1.
\end{equation}
Hence, by \eqref{mars00} and \eqref{mars000}, we conclude \eqref{mars}.
\end{proof}
To establish the next auxiliary results we consider a non increasing cut-off function $\phi\in
\mathcal{C}_0^{\infty}(\mathbb{R}^{n})$ and
\begin{equation}\label{phieps}
\phi_{\varepsilon}(x):=\phi(x/\varepsilon).
\end{equation}
Now we get the following.
\begin{lemma}\label{I1}
Let $\{z_m\}$ be an uniformly bounded sequence in $X^s_{0}(\Omega)$ and $\phi_{\varepsilon}$ the function defined in \eqref{phieps}. Then,
\begin{equation}\label{I11}
{\limm_{\varepsilon\to0}}\limm_{m\to\infty}{\left|\int_{\mathbb{R}^{n}}{z_m(x)(-\Delta)^{s/2}\phi_{\varepsilon}(x)(-\Delta)^{s/2}z_m(x)\, dx}\right|=0.}
\end{equation}
\end{lemma}
\begin{proof}
First of all note that, as a
consequence of the fact that $\{z_m\}$ is uniformly bounded in the reflexive space $X^s_{0}(\Omega)$, say by $M$, we get that there exists $z\in X^s_{0}(\Omega)$, such that, up to a subsequence,
\begin{eqnarray}
\displaystyle z_{m}&\rightharpoonup& z \qquad
\mbox{ weakly in } X^s_{0}(\Omega), \nonumber\\
\displaystyle z_{m}&\to& z \qquad
\mbox{ strongly in } L^{r}(\Omega), \quad  \, 1\leq r < 2^*_{s} \label{estrella1},\\
\displaystyle z_{m}&\to& z \qquad
\mbox{ a.e. in
} \Omega.\nonumber
\end{eqnarray}
Also it is clear that
\begin{equation}\label{I10}
|(-\Delta)^{s/2}\phi_{\varepsilon}(x)|=\varepsilon^{-s}\left|\left((-\Delta)^{s/2}\phi\right)\left(\frac{x}{\varepsilon}\right)\right|\leq C\varepsilon^{-s}.
\end{equation}
Therefore defining
$$I_1:=\left|\int_{\mathbb{R}^{n}}{z_m(x)(-\Delta)^{s/2}\phi_{\varepsilon}(x)(-\Delta)^{s/2}z_m(x)\, dx}\right|,$$
from \eqref{I10} and the fact that $\|z_m\|_{X^s_{0}(\Omega)}<M$, we get
\begin{eqnarray}
I_1&\leq&\|(-\Delta)^{s/2}z_m\|_{L^{2}(\mathbb{R}^{n})}\|z_m(-\Delta)^{s/2}\phi_{\varepsilon}\|_{L^{2}(\Omega)}\nonumber\\
&\leq&M\|(z_m-z)(-\Delta)^{s/2}\phi_{\varepsilon}\|_{L^{2}(\Omega)}+M\|z(-\Delta)^{s/2}\phi_{\varepsilon}\|_{L^{2}(\Omega)}\nonumber\\
&\leq&C\varepsilon^{-s}\|z_m-z\|_{L^{2}(\Omega)}+M\|z(-\Delta)^{s/2}\phi_{\varepsilon}\|_{L^{2}(\Omega)}.\label{I101}
\end{eqnarray}
Since $\|z\|_{X^s_{0}(\Omega)}\leq M$ then $\|z\|_{L^{2^*_s}(\Omega)}\leq C$, that is $z^2\in L^{\frac{n}{n-2s}}(\Omega)$. Hence, for every $\rho>0$ there exits $\eta\in C^{\infty}_{0}(\Omega)$ such that
\begin{equation}\label{I102}
\|z^2-\eta\|_{L^{\frac{n}{n-2s}}(\Omega)}\leq\rho.
\end{equation}
Then, by \eqref{I10}, \eqref{I102} and  H\"{o}lder's inequality with $p=n/n-2s$ we obtain that
\begin{eqnarray}
\|z(-\Delta)^{s/2}\phi_{\varepsilon}\|^2_{L^{2}(\Omega)}&\leq&\int_{\mathbb{R}^{n}}{|z^2(x)-\eta(x)||(-\Delta)^{s/2}\phi_{\varepsilon}(x)|^2\, dx}\nonumber\\
&+&\int_{\mathbb{R}^{n}}{|\eta(x)||(-\Delta)^{s/2}\phi_{\varepsilon}(x)|^2\, dx}\nonumber\\
&\leq&\|z^2-\eta\|_{L^{\frac{n}{n-2s}}(\Omega)}\|(-\Delta)^{s/2}\phi_{\varepsilon}\|^2_{L^{\frac{n}{s}}(\mathbb{R}^{n})}\nonumber\\
&+&\|\eta\|_{L^{\infty}(\Omega)}\|(-\Delta)^{s/2}\phi_{\varepsilon}\|^2_{L^{2}(\mathbb{R}^{n})}\nonumber\\
&\leq&\rho\varepsilon^{-2s}\left(\int_{\mathbb{R}^{n}}{\left|\left((-\Delta)^{s/2}\phi\right)\left(\frac{x}{\varepsilon}\right)\right|^{\frac{n}{s}}\, dx}\right)^{\frac{2s}{n}}\nonumber\\
&+&C\varepsilon^{-2s}\int_{\mathbb{R}^{n}}{\left|\left((-\Delta)^{s/2}\phi\right)\left(\frac{x}{\varepsilon}\right)\right|^{2}\, dx}\nonumber\\
&\leq&\rho\left(\int_{\mathbb{R}^{n}}{|(-\Delta)^{s/2}\phi(z)|^{\frac{n}{s}}\, dz}\right)^{\frac{2s}{n}}\nonumber\\
&+&C\varepsilon^{n-2s}\int_{\mathbb{R}^{n}}{|(-\Delta)^{s/2}\phi(z)|^{2}\, dz}\nonumber\\
&\leq&C\rho+C\varepsilon^{n-2s}.\label{I103}
\end{eqnarray}
Hence, using \eqref{estrella1}, from \eqref{I101}, \eqref{I103} and the fact that $n>2s$, it follows that
$${\limm_{\varepsilon\to0}}\limm_{m\to\infty}{I_1}\leq\limm_{\varepsilon\to0}{C\left(\rho+\varepsilon^{n-2s}\right)^{\frac{1}{2}}}=C\rho^\frac{1}{2}.$$
Since $\rho>0$ is fixed but arbitrarily small, we conclude the proof of Lemma \ref{I1}.
\end{proof}
Also, we have the following.
\begin{lemma}\label{I2}
With the same assumptions of Lemma \ref{I1} we have that
\begin{equation}\label{I21}
{\limm_{\varepsilon\to0}}\limm_{m\to\infty}{\left|\int_{\mathbb{R}^{n}}{(-\Delta)^{s/2}z_m(x)B(z_m,\phi_{\varepsilon})(x)\, dx}\right|=0,}
\end{equation}
where $B$ is defined in \eqref{bilinear}.
\end{lemma}
\begin{proof}
Let
$$I_2:=\left|\int_{\mathbb{R}^{n}}{(-\Delta)^{s/2}z_m(x)B(z_m,\phi_{\varepsilon})(x)\, dx}\right|.$$
Since $\|z_m\|_{X_0^s(\Omega)}\leq M$, then
\begin{eqnarray}
I_2&\leq&M\|B(z_m,\phi_{\varepsilon})\|_{L^{2}(\mathbb{R}^{n})}\nonumber\\
&\leq&M\|B(z_m-z,\phi_{\varepsilon})\|_{L^{2}(\mathbb{R}^{n})}+M\|B(z,\phi_{\varepsilon})\|_{L^{2}(\mathbb{R}^{n})},\label{I22}
\end{eqnarray}
where $z$ is, as in Lemma \ref{I1}, the weak limit of the sequence $\{z_m\}$ in $X^s_{0}(\Omega)$.
We  estimate each of the summands in the previous inequality. Let
\begin{equation}\label{phi0}
\psi(x):=\frac{1}{1+|x|^{n+s}}\quad\mbox{and}\quad\psi_{\varepsilon}(x):=\psi\left(\frac{x}{\varepsilon}\right).
\end{equation}
By Lemma \ref{cero} applied to $\phi$, we note that
\begin{equation}\label{I23}
B(\phi_{\varepsilon},\phi_{\varepsilon})(x) = \varepsilon^{-s}B(\phi,\phi)\left(\frac{x}{\varepsilon}\right)\leq C\varepsilon^{-s}\psi\left(\frac{x}{\varepsilon}\right)= C\frac{\varepsilon^{-s}}{1+\left|\frac{x}{\varepsilon}\right|^{n+s}}\leq C\varepsilon^{-s}.
\end{equation}
Therefore, by  Cauchy-Schwarz inequality and \eqref{I23}, it follows that
\begin{eqnarray}
\, \, \, \, \, \|B(z_m-z,\phi_{\varepsilon})\|^2_{L^{2}(\mathbb{R}^{n})}&\leq&\int_{\mathbb{R}^{n}}{B(z_m-z,z_m-z)(x)B(\phi_{\varepsilon},\phi_{\varepsilon})(x)\, dx}\label{I224}\\
&\leq&C\varepsilon^{-s}\int_{\mathbb{R}^{n}}{B(z_m-z,z_m-z)(x)\, dx}\nonumber\\
&=&C\varepsilon^{-s}\|z_m-z\|^2_{X_0^{\frac{s}{2}}(\Omega)}\nonumber\\
&=&C\varepsilon^{-s}\int_{\mathbb{R}^{n}}{(z_m-z)(x)(-\Delta)^{s/2}(z_m-z)(x)\, dx}\nonumber\\
&\leq&C\varepsilon^{-s}\|z_m-z\|_{L^2(\Omega)}\|(-\Delta)^{s/2}(z_m-z)\|_{L^2(\mathbb{R}^{n})}\nonumber\\
&\leq&C\varepsilon^{-s}\|z_m-z\|_{L^2(\Omega)}.\label{I24}
\end{eqnarray}
On the other hand, for a suitable function $f$, we have that
\begin{eqnarray}
\int_{\mathbb{R}^{n}}{z^2(x)(-\Delta)^{s/2}f(x)\, dx}&=&\int_{\mathbb{R}^{n}}{f(x)(-\Delta)^{s/2}z^2(x)\, dx}\nonumber\\
&=&\int_{\mathbb{R}^{n}}{f(x)\left(2z(x)(-\Delta)^{s/2}z(x)-B(z,z)(x)\right)\, dx}.\label{I2500}
\end{eqnarray}
Then, arguing as in \eqref{I224} and applying \eqref{I2500} with
$f:=\psi_{\varepsilon}(x)$, from \eqref{I23} we get
that
\begin{eqnarray}
\|B(z,\phi_{\varepsilon})\|^2_{L^2(\mathbb{R}^{n})}&\leq&\int_{\mathbb{R}^{n}}{B(z,z)(x)B(\phi_{\varepsilon},\phi_{\varepsilon})(x)\, dx}\nonumber\\
&\leq&C\varepsilon^{-s}\int_{\mathbb{R}^{n}}{B(z,z)(x)\psi_{\varepsilon}(x)\, dx}\nonumber\\
&\leq&C\varepsilon^{-s}\int_{\mathbb{R}^{n}}{\left(-z^2(x)(-\Delta)^{s/2}\psi_{\varepsilon}(x)+2z(x)\psi_{\varepsilon}(x)(-\Delta)^{s/2}z(x)\right)\, dx}\nonumber\\
&:=&I_{2,1}+I_{2,2}.\label{I26}
\end{eqnarray}
We estimate now $I_{2,1}$ and $I_{2,2}$ separately. Let $\rho>0$. By Lemma \ref{cero} applied to $\psi$ and \eqref{I10}, it follows that
\begin{eqnarray}
|I_{2,1}|&\leq&C\varepsilon^{-2s}\int_{\mathbb{R}^{n}}{z^2(x)\left|\left((-\Delta)^{s/2}\psi\right)\left(\frac{x}{\varepsilon}\right)\right|\, dx}\nonumber\\
&\leq&C\varepsilon^{-2s}\int_{\mathbb{R}^{n}}{z^2(x)\psi\left(\frac{x}{\varepsilon}\right)\, dx}\nonumber\\
&\leq&C\varepsilon^{-2s}\int_{\mathbb{R}^{n}}{(z^2-\eta)(x)\psi_{\varepsilon}(x)\, dx}+\varepsilon^{-2s}\int_{\mathbb{R}^{n}}{\eta(x)\psi_{\varepsilon}(x)\, dx}\label{I27},
\end{eqnarray}
where $\eta\in C^{\infty}_{0}(\Omega)$ is the function that satisfies \eqref{I102}. Then from \eqref{I27} we obtain
\begin{eqnarray}
|I_{2,1}|&\leq&C \rho\varepsilon^{-2s}\left\|\psi_{\varepsilon}\right\|_{L^{\frac{n}{2s}}(\mathbb{R}^{n})}+C \varepsilon^{-2s}\|\eta\|_{L^{\infty}(\mathbb{R}^{n})}\left\|\psi_{\varepsilon}\right\|_{L^{1}(\mathbb{R}^{n})}\nonumber\\
&\leq&C \rho\|\psi\|_{L^{\frac{n}{2s}}(\mathbb{R}^{n})}+C \varepsilon^{n-2s}\|\eta\|_{L^{\infty}(\mathbb{R}^{n})}\|\psi\|_{L^{1}(\mathbb{R}^{n})}.\label{I28}
\end{eqnarray}
On the other hand,
\begin{equation}\label{extra}
|I_{2,2}|\leq C\varepsilon^{-s}\|(-\Delta)^{s/2}z\|_{L^2(\mathbb{R}^{n})}\left\|z\psi_{\varepsilon}\right\|_{L^2(\Omega)}\leq C\varepsilon^{-s}\left\|z\psi_{\varepsilon}\right\|_{L^2(\Omega)}.
\end{equation}
Therefore, by \eqref{I102}, we get
\begin{eqnarray}
|I_{2,2}|^2&\leq&C\varepsilon^{-2s}\left(\int_{\Omega}{|(z^2-\eta)(x)|\left|\psi_{\varepsilon}(x)\right|^2\, dx}+\int_{\mathbb{R}^{n}}{\eta\left|\psi_{\varepsilon}(x)\right|^2\, dx}\right)\nonumber\\
&\leq&C\varepsilon^{-2s}\left(\rho \|\psi_{\varepsilon}\|^2_{L^{\frac{n}{s}}(\mathbb{R}^{n})}+\|\eta\|_{L^{\infty}(\mathbb{R}^{n})}\left\|\psi_{\varepsilon}\right\|^2_{L^{2}(\mathbb{R}^{n})}\right)\nonumber\\
&\leq&C\rho\|\psi\|^{2}_{L^{\frac{n}{s}}(\mathbb{R}^{n})}+C\varepsilon^{n-2s}\|\eta\|_{L^{\infty}(\mathbb{R}^{n})}\|\psi\|^2_{L^{2}(\mathbb{R}^{n})}.\label{I29}
\end{eqnarray}
Then, by \eqref{I28} and \eqref{I29}, it follows from \eqref{I26} that
\begin{equation}\label{I210}
\|B(z,\phi_{\varepsilon})\|^2_{L^2(\mathbb{R}^{n})}\leq C\left(\rho+\rho^\frac{1}{2}\right)+{C}\left(\varepsilon^{n-2s}+\varepsilon^{\frac{n-2s}{2}}\right).
\end{equation}
Hence, from \eqref{estrella1}, \eqref{I24} and \eqref{I210}, since $n>2s$, we obtain
\begin{eqnarray}
&&{\limm_{\varepsilon\to0}}\limm_{m\to\infty}\left(\|B(z_m-z,\phi_{\varepsilon})\|^2_{L^{2}(\mathbb{R}^{n})}+\|B(z,\phi_{\varepsilon})\|^2_{L^2(\mathbb{R}^{n})}\right)\nonumber\\
&\leq&{\limm_{\varepsilon\to0}}\, C\left(\rho^{\frac{1}{2}}+\varepsilon^{\frac{n-2s}{2}}\right)\nonumber\\
&=&C\rho^{\frac{1}{2}}.\nonumber
\end{eqnarray}
Thus, since $\rho$ is an arbitrary positive value,
\begin{equation}\label{I211}
{\limm_{\varepsilon\to0}}\limm_{m\to\infty}\Big(
\|B(z_m-z,\phi_{\varepsilon})\|^2_{L^{2}(\mathbb{R}^{n})}+
\|B(z,\phi_{\varepsilon})\|^2_{L^2(\mathbb{R}^{n})}\Big)=0.
\end{equation}
Finally, by \eqref{I22} and \eqref{I211}, we conclude that
$${\limm_{\varepsilon\to0}}\limm_{m\to\infty}{|I_2|}=0.$$
\end{proof}
Now we can prove the principal result of this subsection:
\begin{lemma}\label{strongly_convergent} If $u=0$ is the only
critical point of $\widetilde{\mathcal J}_{s,\,\lambda}$ in
$X^{s}_{0}(\Omega)$, then $\widetilde{\mathcal J}_{s,\,\lambda}$
satisfies a local Palais Smale condition below the critical level
\begin{equation}\label{cstar}
c^*=\frac{s}{n}S(n,s)^{\frac{n}{2s}},
\end{equation}
where $S(n,s)$ is the Sobolev constant defined in \eqref{Ss}.
\end{lemma}
\begin{proof}
Let $\{u_{m}\}$ be a Palais-Smale sequence for $\widetilde{\mathcal J}_{s,\,\lambda}$
verifying
\begin{equation}\label{cc}
\widetilde{\mathcal J}_{s,\,\lambda}(u_m)\to
c_1<{c^*}\qquad\mbox{and}\qquad\widetilde{\mathcal J}_{s,\,\lambda}'(u_m)\to 0.
\end{equation}
{Then, since there exists $M>0$ such that $\|u_m\|_{X_0^s(\Omega)}\leq M$, and, by hypothesis $u=0$ is the unique critical point of $\widetilde{\mathcal J}_{s,\,\lambda}$, it follows that
\begin{eqnarray}
\displaystyle u_{m}&\rightharpoonup& 0 \qquad
\mbox{ weakly in } X^s_{0}(\Omega), \nonumber\\
\displaystyle u_{m}&\to& 0 \qquad
\mbox{ strongly in } L^{r}(\Omega), \quad  1\leq r < 2^*_{s}, \label{estrella0}\\
\displaystyle u_{m}&\to& 0 \qquad
\mbox{ a.e. in
} \Omega.\nonumber
\end{eqnarray}
}
Also, since $u_0$ is a critical point of $\mathcal J_{s,\,\lambda}$, we have that
{
\begin{eqnarray}
\mathcal J_{s,\,\lambda}(z_m)&=&\widetilde{\mathcal J}_{s,\,\lambda}(u_m)+\mathcal J_{s,\,\lambda}(u_0)\nonumber\\
&+&\lambda\int_{\Omega}{\left(\frac{(u_0+(u_m)_+)^{q+1}}{q+1}+u_0^q(u_m-(u_m)_+)-\frac{(u_0+u_m)_{+}^{q+1}}{q+1}\right)\, dx}\nonumber\\
&+&\int_{\Omega}{\left(\frac{(u_0+(u_m)_+)^{2^*_s}}{2^*_s}+u_0^{2^*_s-1}(u_m-(u_m)_+)-\frac{(u_0+u_m)_{+}^{2^*_s}}{2^*_s}\right)\, dx}\nonumber\\
&\leq&\widetilde{\mathcal J}_{s,\,\lambda}(u_m)+\mathcal J_{s,\,\lambda}(u_0), \label{tesis1}
\end{eqnarray}
}
where
\begin{equation}\label{zm}
z_{m}=u_{m}+u_{0}.
\end{equation}
{
Moreover, for every $\varphi\in X_0^s(\Omega)$,
\begin{eqnarray}
\langle\mathcal J_{s,\,\lambda}'(z_{m}),\varphi\rangle&=&\langle\widetilde{\mathcal J}_{s,\,\lambda}'(u_{m}),\varphi\rangle\nonumber\\
&+&\int_{\Omega}{\left(\lambda(u_0+(u_m)_+)^q+(u_0+(u_m)_+)^{2^*_s-1}\right)\varphi\, dx}\nonumber\\
&-&\int_{\Omega}{\left(\lambda(u_0+u_m)_{+}^q+(u_0+u_m)_+^{2^*_s-1}\right)\varphi\, dx}.
\label{tesis20}
\end{eqnarray}
Then, by \eqref{cc},  \eqref{estrella0} and \eqref{tesis20} we obtain that
\begin{equation}\label{tesis2}
{\mathcal J}_{s,\,\lambda}'(z_m)\to 0.
\end{equation}
}
From  \eqref{tesis1} and \eqref{tesis2} we get that the sequence
$\{z_m\}$ is uniformly bounded in $X^s_{0}(\Omega)$. As a
consequence, and the fact that $u=0$ is the unique critical point of $\mathcal J_{s,\,\lambda}$, up to a subsequence, we get that
\begin{eqnarray}
\displaystyle z_{m}&\rightharpoonup& u_0 \qquad
\mbox{ weakly in } X^s_{0}(\Omega), \nonumber\\
\displaystyle z_{m}&\to& u_0 \qquad
\mbox{ strongly in } L^{r}(\Omega), \quad  \, 1\leq r < 2^*_{s}, \label{estrella}\\
\displaystyle z_{m}&\to& u_0 \qquad
\mbox{ a.e. in
} \Omega.\nonumber
\end{eqnarray}

Following \cite{Mejers-Serrin} it is easy to prove that $X_0^s(\Omega)$ could also be defined as the closure of $C^\infty_0(\Omega)$ with respect to the $X_0^s(\Omega)$--norm (see also \cite{fsvnew}). Hence, applying \cite[Theorem~1.5]{pp} we have that there exist an index set $I\subseteq\mathbb{N}$,
a sequence of points $\{ x_k\}_{k\in I}\subset\Omega$, and two sequences of nonnegative real numbers
$\{ \mu_k\}_{k\in I},\,\{ \nu_k\}_{k\in I}$, such that
\begin{equation}\label{PS1}
|(-\Delta)^{s/2}{(z_m)_+}|^2 \rightarrow \mu \geq
|(-\Delta)^{s/2}u_0|^2  + \sum_{k\in I}
\mu_{k}\delta_{x_{k}}.
\end{equation}
Moreover,
\begin{equation}\label{PS2}
|{(z_{m})_+}|^{2^*_{s}} \to
\nu=|u_{0}|^{2^*_{s}} + \sum_{k\in I}
\nu_{k}\delta_{x_{k}},
\end{equation}
in the sense of measures, with
\begin{equation}\label{nuk}
\nu_{k}\leq S(n,s)^{-\frac{2^*_{s}}{2}}\mu_{k}^{\frac{2^*_{s}}{2}}\,\,\,\,\, \mbox{for every}\,\, k\in I\,.
\end{equation}
Here $\delta_{x_k}$ denotes the Dirac delta at $x_k$, while $S(n,s)$ is the constant given in \eqref{Ss}\,.
We fix $k_0\in I$, and we consider $\phi\in
\mathcal{C}_0^{\infty}(\mathbb{R}^{n})$ a nonincreasing cut-off
function satisfying
\begin{equation}\label{phi}
\phi = 1\, \mbox{ in }\, B_{1}(x_{k_{0}})\quad\mbox{and}\quad
\phi = 0\, \mbox{ in }\, B_{2}(x_{k_{0}})^{c}.
\end{equation}
Set now
\begin{equation}\label{phieps1}
\phi_{\varepsilon}(x)=\phi(x/\varepsilon), \quad x\in \mathbb R^n.
\end{equation}
Taking the  derivative of the identity given in \eqref{motorista}, see also \cite[Lemma 16]{sv}, for any $u, \varphi \in X_0^s(\Omega)$ we obtain that
\begin{equation}\label{look}
\int_{\mathbb{R}^{n}\times \mathbb R^n}{\frac{(u(x)-u(y))(\varphi(x)-\varphi(y))}{|x-y|^{n+2s}}\, dx\, dy}=\int_{\mathbb{R}^{n}}{\varphi(x)(-\Delta)^su(x)\, dx}.
\end{equation}
Then, using
$\phi_{\varepsilon}{(z_{m})_+}$ as a test function in (\ref{tesis2}), by \eqref{look}, {and the fact that
$$\int_{\mathbb{R}^{n}}{(\phi_{\varepsilon}(z_m)_+)(-\Delta)^sz_m\, dx}\geq\int_{\mathbb{R}^{n}}{(\phi_{\varepsilon}(z_m)_+)(-\Delta)^s(z_m)_+\, dx},$$
we have that}
\begin{eqnarray*}
0&{\geq}&\limm_{m\to\infty}\left(\int_{\mathbb{R}^{n}}{(\phi_{\varepsilon}{(z_m)_+})(-\Delta)^s{(z_m)_+}\, dx}\right.\\
&-&\left.\left(\lambda\int_{B_{2\varepsilon}(x_{k_0})}{{((z_m)_+)}^{q+1}\phi_{\varepsilon} \,dx}+\int_{B_{2\varepsilon}(x_{k_0})}{{((z_m)_+)}^{2^*_{s}}\phi_{\varepsilon} \,dx}\right)\right).
\end{eqnarray*}
Hence,
\begin{eqnarray*}
& &\quad  \limm_{m\to\infty}\left(\int_{\mathbb{R}^{n}}{{(z_m)_+}(x)(-\Delta)^{s/2}{(z_m)_+}(x)(-\Delta)^{s/2}\phi_{\varepsilon}(x)\, dx}\right.\\
& &-2\left. \int_{\mathbb{R}^{n}}{(-\Delta)^{s/2}{(z_m)_+}(x)\int_{\mathbb{R}^{n}}{\frac{(\phi_{\varepsilon}(x)-\phi_{\varepsilon}(y))({(z_m)_+}(x)-{(z_m)_+}(y))}{|x-y|^{n+s}}\, dx\, dy}}\right)\\
& &\leq \limm_{m\to\infty}\left(\lambda\int_{B_{2\varepsilon}(x_{k_0})}{{((z_m)_+)}^{q+1}
\phi_{\varepsilon} \,
dx}+\int_{B_{2\varepsilon}(x_{k_0})}{{((z_m)_+)}^{2^*_{s}}
\phi_{\varepsilon} \,
dx}\right.\\
& &-\left.\int_{B_{2\varepsilon}(x_{k_0})}{((-\Delta)^{s/2}{(z_m)_+})^2\phi_{\varepsilon}\, dx}\right).
\end{eqnarray*}
Therefore, by (\ref{estrella}), (\ref{PS1}) and (\ref{PS2}) we get
\begin{eqnarray}
& &\quad \limm_{\varepsilon\to0}\limm_{m\to\infty}\left(\int_{\mathbb{R}^{n}}{{(z_m)_+}(x)(-\Delta)^{s/2}{(z_m)_+}(x)(-\Delta)^{s/2}\phi_{\varepsilon}(x)\, dx}\right.\nonumber\\
& &-2\left. \int_{\mathbb{R}^{n}}{(-\Delta)^{s/2}{(z_m)_+}(x)\int_{\mathbb{R}^{n}}{\frac{(\phi_{\varepsilon}(x)-\phi_{\varepsilon}(y))({(z_m)_+}(x)-{(z_m)_+}(y))}{|x-y|^{n+s}}\, dx\, dy}}\right)\nonumber\\
 & & \; \leq \limm_{\varepsilon\to0}\left(\lambda\int_{B_{2\varepsilon}(x_{k_0})}u_{0}^{q+1}
\phi_{\varepsilon} \,
dx + \int_{B_{2\varepsilon}(x_{k_0})}\phi_{\varepsilon} \,
d\nu-\int_{B_{2\varepsilon}(x_{k_0})}\phi_{\varepsilon}
\, d\mu\right).\label{lions1}
\end{eqnarray}
Since $\phi$ is a regular function with compact support is clear that satisfies the hypothesis of Lemma \ref{cero}. Therefore, by Lemma \ref{I1} and Lemma \ref{I2} applied to the sequence $\{{(z_m)_+}\}$, it follows that the left hand side of \eqref{lions1} goes to zero. That is, we obtain that
$$\limm_{\varepsilon \to 0}
\left(\int_{B_{2\varepsilon}(x_{k_0})}\phi_{\varepsilon} \,
d\nu+\lambda\int_{B_{2\varepsilon}(x_{k_0})}u_{0}^{q+1}
\phi_{\varepsilon} \,
dx-\int_{B_{2\varepsilon}(x_{k_{0}})}\phi_{\varepsilon}
\, d\mu \right)
=\nu_{k_{0}}-\mu_{k_{0}}{\geq}0.$$
Thus, from \eqref{nuk}, we have that either $\nu_{k_0}=0$ or
\begin{equation}\label{ddorm}
\nu_{k_0}\geq S(n,s)^{\frac{n}{2s}}.
\end{equation}
\\Suppose now that $\nu_{k_{0}}\neq 0$. {By \eqref{tesis1}, \eqref{tesis2} and \eqref{ddorm} we obtain that}

\begin{eqnarray*}
c_1+\mathcal J_{s,\,\lambda}(u_0)&{\geq}&\limm_{m\to\infty}{\left(\mathcal J_{s,\,\lambda}(z_{m})- \frac{1}{2}\langle
\mathcal J_{s,\,\lambda}'(z_{m}),z_{m}\rangle\right)}\\
&\geq& \lambda
\left(\frac{1}{2}-\frac{1}{q+1}
 \right)\int_{\Omega}{u_{0}^{q+1}dx}+\frac{s}{n}\int_{\Omega}{u_{0}^{2^*_{s}}dx} +\frac{s}{n}\nu_{k_0}\\
 &\geq &\mathcal J_{s,\,\lambda}(u_{0})+\frac{s}{n}S(n,s)^{\frac{n}{2s}}\nonumber\\
 &=&\mathcal J_{s,\,\lambda}(u_{0})+c^*.
\end{eqnarray*}
This is a contradiction with  \eqref{cc}. Since $k_0$ was arbitrary, we deduce that
$\nu_{k}= 0$ for all $k\in I$. As a consequence, we obtain that {$(u_{m})_{+}\to 0$} in
$L^{2^*_{s}}(\Omega)$. Note that, since $u_m$ is equal to zero outside $\Omega$, indeed we have that {$(u_{m})_{+}\to 0$} in
$L^{2^*_{s}}(\mathbb{R}^{n})$. This implies convergence of ${\lambda ((u_{m})_{+})^{q}+((u_{m})_{+})^{2^*_s-1}}$
in $L^{\frac{2n}{n+2s}}(\mathbb{R}^{n})$. Finally,  using the continuity of the inverse operator
$(-\Delta)^{-s}$, we obtain strong convergence of $u_{m}$ in
$X^s_{0}(\Omega)$.
\end{proof}
\subsection{Proof of statement $(4)$ of Theorem~\ref{lapfra0q}}\label{sec:dimtheorem11}
\
In Lemma~\ref{strongly_convergent} we have proved that if $u\equiv 0$ is the only critical point of the functional~$\widetilde{\mathcal J}_{s,\,\lambda}$, then $\widetilde{\mathcal J}_{s,\,\lambda}$ verifies the Palais--Smale condition at any level $c_1<c^*$, where $c^*$ is the critical level defined in \eqref{cstar}.

Now, we want to show that we can obtain a local (PS)$_{c}$--sequence for $\widetilde{\mathcal J}_{s,\,\lambda}$ under the critical level $c^*$.
For this, assume, without loss of generality, that $0\in \Omega$.
By \cite{cotsiolis} (see also \cite{colorado1, lieb}) the infimum in \eqref{Ss} is attained at the function
\begin{equation}\label{talenti}
u_{\varepsilon}(x)=\frac{\varepsilon^{(n-2s)/2}}{(|x|^2+\varepsilon^2)^{(n-2s)/2}},\,\,\, \varepsilon>0,
\end{equation}
that is
\begin{equation}\label{ojo}
\|(-\Delta)^{s/2}u_{\varepsilon}\|^{2}_{L^{2}(\mathbb{R}^{n})}
=\int_{\mathbb{R}^{n}\times\mathbb{R}^{n}}{\frac{|u_{\varepsilon}(x)-u_{\varepsilon}(y)|^2}{|x-y|^{n+2s}}\, dx\, dy}=S(n,s)\|u_{\varepsilon}\|_{L^{2^*_{s}}(\mathbb{R}^n)}^{2}.
\end{equation}

Also, let us introduce a cut-off function $\phi_{0}\in
C^{\infty}(\mathbb{R})$, non increasing and satisfying
$$
\phi_{0}(t) = \left\{\begin{array}{ll}
1 & \mbox{ if } 0\leq t \leq \frac{1}{2},\\
0 & \mbox{ if } t\geq 1.
\end{array}\right.
$$
For a fixed $r>0$ small enough such that
$\overline{B}_{r}\subset\Omega$, set
$\phi(x)=\phi_{r}(x)=\phi_{0}(\frac{|x|}{r})$ and consider the family of {non negative} truncated functions
\begin{equation}\label{etaeps}
\eta_{\varepsilon}(x)= \frac{\phi  u_{\varepsilon}(x)}{\|\phi
u_{\varepsilon}\|_{L^{2^*_{s}}(\Omega)}}\in X_0^{s}(\Omega).
\end{equation}
Then, we have the following.
\begin{lemma}\label{underlevel}
There exists $\varepsilon
>0$ small enough such that
\begin{equation}\label{lowlevel}
\displaystyle \sup_{t\geq 0}\widetilde{\mathcal J}_{s,\,\lambda}(t\eta_{\varepsilon}) <
{c^*}.
\end{equation}
\end{lemma}
\begin{proof}
We follow the proof of \cite[Lemma 4.4]{ABC} (see also \cite[Lemma~3.9]{colorado}).
\\Assume $n\ge 4s$. Since
\begin{equation}\label{estimate-ABC1}
(a+b)^p\geq a^p + b^p + \mu a^{p-1}b,\, \mbox{ for some $\mu>0$ and every $a,b \geq 0$, $p>1,$}
\end{equation}
then the function $G_\lambda$ defined in \eqref{G}, satisfies
\begin{equation}\label{estimatesG}
G_\lambda(u) \geq \frac{1}{2^*_{s}}{(u_+)}^{2^*_{s}}+\frac{\mu}{2}{(u_+)}^2
u_{0}^{2^*_{s} -2}.
\end{equation}
Therefore,
$$
\widetilde{\mathcal J}_{s,\,\lambda}(t\eta_{\varepsilon}) \leq
\frac{t^2}{2}\|\eta_{\varepsilon}\|_{X^s_{0}(\Omega)}^{2}
- \frac{t^{2^*_{s}}}{2^*_{s}}-\frac{t^2}{2}\mu
\int_{\Omega}u_{0}^{2^*_{s}-2}\eta_{\varepsilon}^2dx.
$$
Since $u_{0}\ge a_0>0$ in $\supp(\eta_{\varepsilon})$ we get, for any $t\geq 0$ and $\varepsilon>0$ small enough,
\begin{equation}\label{extra1}
\widetilde{\mathcal J}_{s,\,\lambda}(t\eta_{\varepsilon})\leq
\frac{t^2}{2}\|\eta_{\varepsilon}\|_{X^s_{0}(\Omega)}^{2}-\frac{t^{2^*_{s}}}{2^*_{s}}-\frac{t^2}{2}\widetilde{\mu}\|\eta_{\varepsilon}\|_{L^{2}(\Omega)}^{2}.
\end{equation}
Moreover, since $\|u_{\varepsilon}\|_{L^{2^*_{s}}(\mathbb{R}^{n})}$ is independent of $\varepsilon$,
by \cite[Proposition~21]{servadeivaldinociBN} we have
\begin{eqnarray}
\|\eta_{\varepsilon}\|_{X^s_{0}(\Omega)}^{2}
&=&\frac{\|\phi u_{\varepsilon}\|_{X_0^s(\Omega)}^{2}}{\|\phi u_{\varepsilon}\|^{2}_{L^{2^*_{s}}(\Omega)}}\nonumber\\
&\leq&\frac{\displaystyle\int_{\mathbb{R}^{n}\times\mathbb{R}^{n}}{\frac{|u_{\varepsilon}(x)-u_{\varepsilon}(y)|^2}{|x-y|^{n+2s}}\, dx\, dy}}{\|\phi u_{\varepsilon}\|^{2}_{L^{2^*_{s}}(\Omega)}}+O(\varepsilon^{n-2s})\nonumber\\
&=&S(n,s)+O(\varepsilon^{n-2s}).\label{estimatesOld}
\end{eqnarray}
Furthermore, by \cite[Lemma~3.8]{colorado} (see also \cite[Proposition~22]{servadeivaldinociBN}) it follows that
\begin{equation}\label{estimatesOld22}
\|\eta_{\varepsilon}\|_{L^{2}(\Omega)}^{2}\geq \left\{
\begin{array}{ll}
C\varepsilon^{2s} &\mbox{ if } n> 4s,\\
C\varepsilon^{2s} \log(1\slash \varepsilon) &\mbox{ if }
n=4s.
\end{array}\right.
\end{equation}
Therefore, from \eqref{extra1}, \eqref{estimatesOld} and \eqref{estimatesOld22}, we get
\begin{equation}\label{Jg}
\widetilde{\mathcal J}_{s,\,\lambda}(t\eta_{\varepsilon})\leq\frac{t^2}{2}(S(n,s)+C\varepsilon^{n-2s})-\frac{t^{2^*_{s}}}{2^*_{s}}-\frac{t^2}{2}\widetilde{C}\varepsilon^{2s}:=g(t),
\end{equation}
with $\widetilde{C}>0$. Since $\displaystyle \limm_{t\to\infty}
g(t) = -\infty$, then
$\displaystyle\sup_{t\geq 0}
g(t)$ is attained at
some $t_{\varepsilon,\lambda}:=t_{\varepsilon}\geq0$. If $t_\varepsilon= 0$, then
$$\sup_{t\geq 0} \widetilde{\mathcal J}_{s,\,\lambda}(t \eta_\varepsilon)\leq \sup_{t\geq 0} g(t)=g(0)=0$$ for any $0<\lambda<\Lambda$ and \eqref{lowlevel} is trivially verified. Now, we suppose that $t_\varepsilon>0$. Differentiating the above function $g(t)$, we obtain that
\begin{equation}\label{difJ}
0=g'(t_{\varepsilon})=
t_{\varepsilon}(S(n,s)+C\varepsilon^{n-2s})-t_{\varepsilon}^{2^*_{s}
-1}-t_{\varepsilon}\widetilde{C}\varepsilon^{2s},
\end{equation}
which implies
\begin{equation}\label{fly}
t_{\varepsilon} \leq
(S(n,s)+C\varepsilon^{n-2s})^{\frac{1}{2^*_{s}-2}}.
\end{equation}
Also we have, for $\varepsilon>0$ small enough,
\begin{equation}\label{teps}
t_\varepsilon\ge c>0.\end{equation}
Indeed from \eqref{difJ} we get
$$
t_\varepsilon^{2^*_{s}-2}=S(n,s)+C\varepsilon^{n-2s}-\widetilde{C}\varepsilon^{2s}\geq c>0,
$$
provided $\varepsilon$ is small enough. Moreover, the function
$$
t \mapsto
\frac{t^2}{2}(S(n,s)+C\varepsilon^{n-2s})
- \frac{t^{2^*_{s}}}{2^*_{s}}
$$
is increasing on
$[0,(S(n,s)+C\varepsilon^{n-2s})^{\frac{1}{2^*_{s}-2}}]$.
Whence, by \eqref{fly} and \eqref{teps}, we obtain
$$
\displaystyle \sup_{t\geq 0}
g(t)=g(t_{\varepsilon})\leq
\frac{s}{n}(S(n,s)+C\varepsilon^{n-2s})^{\frac{n}{2s}}-
\overline{C}\varepsilon^{2s},
$$
for some $\overline{C}>0$. Therefore, by \eqref{Jg}, for $n>4s$, we get that
\begin{equation}\label{estimateJ}
\sup_{t\geq0}\widetilde{\mathcal J}_{s,\,\lambda}(t\eta_{\varepsilon})\leq g(t_{\varepsilon})\leq\frac{s}{n}S(n,s)^{\frac{n}{2s}}+C\varepsilon^{n-2s}-\overline{C}\varepsilon^{2s}<\frac{s}{n}S(n,s)^{\frac{n}{2s}}=c^*.
\end{equation}
If $n=4s$ the same conclusion follows.

The last case $2s <n<4s$ follows by using the estimate \eqref{estimate-ABC1} which gives
\begin{equation}\label{estimatesG2}
G_\lambda(u) \geq \frac{1}{2^*_{s}}{(u_+)}^{2^*_{s}}+\frac{\mu}{2^*_{s}-1}u_0 {(u_+)}^{2^*_{s} -1}.
\end{equation}
Then, \eqref{estimatesG2} jointly with the inequality (3.28) of \cite{colorado}, instead of \eqref{estimatesOld22}, and arguing in a similar way as above, finish the proof.

\end{proof}
To complete the existence of the second solution, that is statement $(4)$
in Theorem \ref{lapfra0q}, in view of the previous results, we look for a path with energy below the critical level $c^*$. Let us fix $\lambda\in (0, \Lambda)$. We consider $M_\varepsilon>0$ large enough so that $\widetilde{\mathcal J}_{s,\,\lambda}(M_\varepsilon \eta_\varepsilon)<\widetilde{\mathcal J}_{s,\,\lambda}(0)$.  Note that such $M_\varepsilon$ exists, since ${\displaystyle \lim_{t\to\infty}\widetilde{\mathcal J}_{s,\,\lambda}(t\eta_\varepsilon)=-\infty}$. Also, by Lemma \ref{lemma:minimo}, there exists $\alpha>0$ such that if $\|u\|_{X_0^{s}(\Omega)}=\alpha$, then $\widetilde{\mathcal J}_{s,\,\lambda}(u)\geq \widetilde{\mathcal J}_{s,\,\lambda}(0)$.
We define
$$
\Gamma_\varepsilon=\{ \gamma\in\mathcal{C}([0,1],X_0^s(\Omega)) :\: \gamma(0)=0,\, \gamma(1)=M_\varepsilon \eta_\varepsilon\},
$$
and the minimax value
\begin{equation}\label{cee}
c_\varepsilon=\inf_{\gamma\in \Gamma_\varepsilon}\sup_{0\le t\le 1}\widetilde{\mathcal J}_{s,\,\lambda}(\gamma(t)).
\end{equation}
By the arguments above, $c_\varepsilon\ge\widetilde{\mathcal J}_{s,\,\lambda}(0)$. Also, by Lemma \ref{underlevel}, for $\varepsilon\ll 1$ we obtain that
$$c_\varepsilon\le \sup_{0\le t\le 1}\widetilde{\mathcal J}_{s,\,\lambda}(t M_{\varepsilon}\eta_\varepsilon)=\sup_{t\ge 0}\widetilde{\mathcal J}_{s,\,\lambda}(t \eta_\varepsilon)<c^*.$$
Therefore, by Lemma \ref{strongly_convergent} and the MPT \cite{ar} if $c_\varepsilon>\widetilde{\mathcal J}_{s,\,\lambda}(0)$, or the corresponding refinement given in \cite{GhP} if the minimax level is equal to $\widetilde{\mathcal J}_{s,\,\lambda}(0)$, we obtain the existence of a non-trivial solution of $(\widetilde{P}_{\lambda})$, provided $u\equiv 0$ is its unique solution. Of course this is a contradiction. Thus, $\widetilde{\mathcal J}_{s,\,\lambda}$ admits a critical point $\tilde u$ different from the trivial function. As a consequence, $u=u_0+\tilde u$ is a solution, different of $u_0$, of problem $(P_{\lambda})$. This concludes the proof of Theorem~\ref{lapfra0q}.

\section{The critical and convex case $q>1$}\label{sec:q>1}
In this section we discuss the problem $(P_{\lambda})$ in the convex setting~$q>1$. Here, we argue essentially as in \cite{sY, sBNRES, servadeivaldinociBN, servadeivaldinociBNLOW, servadeivaldinociCFP}, where the authors studied the linear case $q=1$ using again variational techniques. With respect to the case $q=1$, there are some extra difficulties to prove the $(PS)_c$ condition and to obtain the estimates of the Mountain Pass critical value.
First of all it is easy to check the good geometry of the functional. That is we have the following.
\begin{proposition}\label{MPG}
Assume $\lambda>0$ and $1<q<2^{*}_{s}-1$. Then, there exist $\alpha>0$ and $\beta>0$ such that
\begin{itemize}
\item[a)]{for any $u\in X^{s}_0(\Omega)$ with
$||u||_{X^{s}_0(\Omega)}=\alpha$  one has that $\mathcal J_{s,\,\lambda}(u)\geq\beta$,}
\item[b)]{there exists a positive function $e\in X^{s}_0(\Omega)$ so that $||e||_{X^{s}_0(\Omega)}>\alpha$ and $\mathcal J_{s,\,\lambda}(e)<\beta$.}
\end{itemize}
\end{proposition}
\begin{proof}
\begin{itemize}
\item[a)]{By the Sobolev embedding theorem, since $q+1<2^{*}_{s}$, it can be easily seen that
$$\mathcal J_{s,\,\lambda}(u)\geq g(||u||_{X^{s}_0(\Omega)}),$$
where $g(t)=C_1t^2-\lambda C_2t^{q+1}- C_3t^{2^{*}_{s}}$, for some positive constants $C_1, C_2$ and $C_3$.
Therefore, there will exist $\alpha>0$ such that $\beta:=g(\alpha)>0$. Then, $\mathcal J_{s,\,\lambda}(u)\ge\beta$ for $u\in
X^{s}_0(\Omega)$ with $||u||_{X^{s}_0(\Omega)}=\alpha$.
}
\item[b)]{
Fix a positive function $u_0\in X^{s}_0(\Omega)$ such that $||u_0||_{X^{s}_0(\Omega)}=1$ and consider $t>0$. 
Since $2^{*}_{s}>2$, it follows that
$$\limm_{t\rightarrow\infty}\mathcal J_{s,\,\lambda}(t u_0)=-\infty.$$
Then, there exists $t_0$ large enough, such that for $e:=t_0u_0$, we get that $||e||_{X^{s}_0(\Omega)}>\alpha$ and $\mathcal J_{s,\,\lambda}(e)<\beta$.}
\end{itemize}
\end{proof}

By a similar argument, it follows that
\begin{equation}\label{limneg}
\limm_{t\rightarrow 0^+}\mathcal J_{s,\,\lambda}(t u_0)= 0.
\end{equation}

Let us check now that we have the compactness properties of~$\mathcal J_{s,\,\lambda}$\,.

\subsection{The Palais--Smale condition for~$\mathcal J_{s,\,\lambda}$}\label{sec:compactness}
In this subsection we show that the functional~$\mathcal J_{s,\,\lambda}$ satisfies the Palais--Smale condition in a suitable energy range involving the best fractional critical Sobolev constant~$S(n,s)$ given in \eqref{Ss}, that is we prove the following.

\begin{proposition}\label{propPS}
Let $\lambda>0$ and $1<q<2^{*}_{s}-1$.

Then, the functional~$\mathcal J_{s,\,\lambda}$ satisfies the (PS)$_{c_2}$ condition provided $c_2<c^*$\,, where $c^*$ is given in \eqref{cstar}.
\end{proposition}
\begin{proof}
Let $\{u_m\}$ be a (PS)$_{c_2}$--sequence for $\mathcal J_{s,\,\lambda}$ in $X_0^s(\Omega)$, that is
\begin{equation}\label{Jc0}
\mathcal J_{s,\,\lambda}(u_m)\to c_2
\end{equation}
and
\begin{equation}\label{J'00}
\mathcal J_{s,\,\lambda}'(u_m)\to 0.
\end{equation}
First of all we get that $\{u_m\}$ is bounded in $X_0^s(\Omega)$. Indeed by \eqref{Jc0} and \eqref{J'00}, there exists $M>0$ such that
\begin{equation}\label{acot}
\|u_m\|_{X_0^s(\Omega)}\leq M.
\end{equation}

In order to prove our result we proceed by steps.

\medskip

\begin{claim}\label{claim3}
There exists $u_\infty\in X_0^s(\Omega)$ such that $\langle \mathcal J_{s,\,\lambda}'(u_\infty), \varphi\rangle=0$ for any $\varphi \in X_0^s(\Omega)$\,.
\end{claim}
\begin{proof}
By \eqref{acot} and the fact that $X_0^s(\Omega)$ is a reflexive space, up to a
subsequence, still denoted by $u_m$, there exists $u_\infty \in X_0^s(\Omega)$
such that~$u_m\rightharpoonup u_\infty$ weakly in~$X_0^s(\Omega)$, that is
\begin{equation}\label{convergenze0}
\begin{aligned}
 & \int_{\RR^n\times \RR^n}\frac{\big(u_m(x)-u_m(y)\big)\big(\varphi(x)-\varphi(y)\big)}{|x-y|^{n+2s}}\,dx\,dy \to \\
& \qquad \qquad
\int_{\RR^n\times \RR^n}\frac{\big(u_\infty(x)-u_\infty(y)\big)\big(\varphi(x)-\varphi(y)\big)}{
|x-y|^{n+2s}}\,dx\,dy  \quad \mbox{for any}\,\, \varphi\in X_0^s(\Omega).
\end{aligned}
\end{equation}
Moreover, we have
\begin{eqnarray}
\displaystyle u_{m}&\rightharpoonup& u_{\infty} \qquad
\mbox{ weakly in } L^{2^*_s}(\Omega), \label{convergenzaweak}\\
\displaystyle u_{m}&\to& u_{\infty} \qquad
\mbox{ strongly in } L^{r}(\Omega), \quad  \, 1\leq r < 2^*_{s} \label{convergenze0bis},\\
\displaystyle u_{m}&\to& u_{\infty} \qquad
\mbox{ a.e. in
} \Omega.\label{convergenze0ter}
\end{eqnarray}
Hence, taking the limit when $m\to\infty$, by \eqref{J'00}, \eqref{convergenze0}-\eqref{convergenze0ter} we conclude
\begin{eqnarray*}
\int_{\RR^n\times \RR^n}\frac{\left(u_\infty(x)-u_\infty(y)\right) \left(\varphi(x)-\varphi(y)\right)}{|x-y|^{n+2s}}\,dx\,dy&=&\lambda \int_\Omega {((u_\infty)_+)}^q\varphi\,dx\\
&+& \int_\Omega {((u_\infty)_+)}^{2^*_{s}-1}\varphi\,dx,
\end{eqnarray*}
for any $\varphi \in X_0^s(\Omega)$.
\end{proof}

\begin{claim}\label{step4}
The following equality holds:
$$\mathcal J_{s,\,\lambda}(u_m)  =\mathcal J_{s,\,\lambda}(u_\infty) + \frac 1 2 \|u_m-u_{\infty}\|^{2}_{X_0^s(\Omega)}-\frac{1}{2^*_{s}}\int_\Omega |{(u_m)_+}(x)-{(u_\infty)_+}(x)|^{2^*_{s}}dx+\textit{o}(1).
$$
\end{claim}

\begin{proof}
First of all, we observe that by \eqref{acot} and the Sobolev embedding theorem, the sequence $u_m$ is bounded in $X_0^s(\Omega)$ and in $L^{2^*_{s}}(\Omega)$. Hence, since \eqref{convergenze0bis} and \eqref{convergenze0ter} hold true, by the Brezis-Lieb Lemma (see \cite[Theorem~1]{bl}), we get
\begin{equation}\label{bl1}
\|u_m\|^2_{X_0^s(\Omega)}=\|u_m-u_{\infty}\|^{2}_{X_0^s(\Omega)}+\|u_\infty\|^2_{X_0^s(\Omega)}+ \textit{o}(1),
\end{equation}
\begin{equation}\label{bl2}
\int_\Omega|{(u_m)_+}|^{2^*_{s}}\,dx= \int_\Omega |{(u_m)_+}(x)-{(u_\infty)_+}(x)|^{2^*_{s}}\,dx+ \int_\Omega|{(u_\infty)_+}|^{2^*_{s}}\,dx+ \textit{o}(1)
\end{equation}
and
\begin{equation}\label{convq+1}
\|{(u_m)_+}\|_{L^{q+1}(\Omega)}\to \|{(u_\infty)_+}\|_{L^{q+1}(\Omega)}.
\end{equation}
Therefore, by \eqref{bl1}, \eqref{bl2} and \eqref{convq+1} we deduce that
\begin{eqnarray*}
\mathcal J_{s,\,\lambda}(u_m)&=&\frac 1 2 \|u_m-u_{\infty}\|^{2}_{X_0^s(\Omega)}+ \frac 1 2\|u_{\infty}\|^{2}_{X_0^s(\Omega)}\\
&-&\frac{\lambda}{q+1} \int_\Omega {((u_\infty)_+)}^{q+1} dx-\frac{1}{2^*_{s}}\int_\Omega |{(u_m)_+}(x)-{(u_\infty)_+}(x)|^{2^*_{s}}\,dx\\
&-&\frac{1}{2^*_{s}} \int_\Omega {((u_\infty)_+)}^{2^*_{s}}\,dx+\textit{o}(1)\\
&=&\mathcal J_{s,\,\lambda}(u_\infty)+\frac 1 2 \|u_m-u_{\infty}\|^{2}_{X_0^s(\Omega)}\\
&-&\frac{1}{2^*_{s}}\int_\Omega |{(u_m)_+}(x)-{(u_\infty)_+}(x)|^{2^*_{s}}\,dx+\textit{o}(1),
\end{eqnarray*}
which gives the desired assertion.
\end{proof}

\begin{claim}\label{step5}
The following estimate holds:
\begin{eqnarray*}
\|u_m-u_{\infty}\|^{2}_{X_0^s(\Omega)} &=& \int_\Omega |{(u_m)_+}(x)-{(u_\infty)_+}(x)|^{2^*_{s}}dx+\textit{o}(1)\\
&\leq&\int_\Omega |(u_m)(x)-(u_\infty)(x)|^{2^*_{s}}dx+\textit{o}(1).
\end{eqnarray*}
\end{claim}

\begin{proof}
Note that, as a consequence of \eqref{convergenzaweak} and \eqref{bl2}, we get
\begin{eqnarray}
&&\int_\Omega\left({((u_m)_+)}^{2^*_{s}-1}(x)-{((u_\infty)_+)}^{2^*_{s}-1}(x)\right)\left(u_m(x)-u_\infty(x)\right)\,dx\nonumber\\
&=& \int_\Omega {((u_m)_+)}^{2^*_{s}}\,dx -\int_\Omega {((u_\infty)_+)}^{2^*_{s}-1}u_m\,dx\nonumber\\
&-&\int_\Omega {((u_m)_+)}^{2^*_{s}-1}u_\infty\,dx+ \int_\Omega  {((u_\infty)_+)}^{2^*_{s}}\,dx\nonumber\\
&=& \int_\Omega {((u_m)_+)}^{2^*_{s}}\,dx- \int_\Omega  {((u_\infty)_+)}^{2^*_{s}}\,dx + \textit{o}(1)\nonumber\\
&=& \int_\Omega |{(u_m)_+}(x)-{(u_\infty)_+}(x)|^{2^*_{s}}\,dx+ \textit{o}(1).\label{ujuinfty}
\end{eqnarray}
Furthermore, \eqref{convergenze0bis} and \eqref{convq+1} give
\begin{eqnarray}
&&\int_\Omega\left({((u_m)_+)}^{q}(x)-{((u_\infty)_+)}^{q}(x)\right)\left(u_m(x)-u_\infty(x)\right)\,dx\nonumber\\
&=&\int_\Omega {((u_m)_+)}^{q+1}\,dx -\int_\Omega {((u_\infty)_+)}^{q}u_m\,dx\nonumber\\
&-&\int_\Omega {((u_m)_+)}^{q}u_\infty\,dx+ \int_\Omega  {((u_\infty)_+)}^{q+1}\,dx\nonumber\\
&=&\textit{o}(1).\label{ujuinftyq}
\end{eqnarray}
Then, by \eqref{J'00}, Claim~\ref{claim3}, \eqref{ujuinfty} and \eqref{ujuinftyq}, we conclude that
\begin{eqnarray*}
\textit{o}(1)&=&\langle\mathcal J'_{s,\,\lambda}(u_m), u_m-u_\infty\rangle\\
&=&\langle\mathcal J'_{s,\,\lambda}(u_m)-\mathcal J'_{s,\,\lambda}(u_\infty), u_m-u_\infty\rangle\\
&=&\|u_m-u_{\infty}\|^{2}_{X_0^s(\Omega)}-\lambda\int_\Omega\left({((u_m)_+)}^{q}(x)-{((u_\infty)_+)}^{q}(x)\right)(u_m(x)-u_\infty(x))\,dx\\
&-&\int_\Omega\left({((u_m)_+)}^{2^*_{s}-1}(x)-{((u_\infty)_+)}^{2^*_{s}-1}(x)\right)(u_m(x)-u_\infty(x))\,dx\\
&=& \|u_m-u_{\infty}\|^{2}_{X_0^s(\Omega)}- \int_\Omega |{((u_m)_+)}(x)-{((u_\infty)_+)}(x)|^{2^*_{s}}\,dx+ \textit{o}(1).
\end{eqnarray*}
\end{proof}

Now, we can finish the proof of Proposition~\ref{propPS}\,.

\medskip

By Claim~\ref{step5} we know that
\begin{equation}\label{sstu}
\frac{1}{2}\|u_m-u_{\infty}\|^{2}_{X_0^s(\Omega)}-\frac{1}{2^*_{s}}\int_\Omega |{((u_m)_+)}(x)-{((u_\infty)_+)}(x)|^{2^*_{s}}dx =\frac{s}{n}\|u_m-u_{\infty}\|^{2}_{X_0^s(\Omega)}+\textit{o}(1).
\end{equation}
Then,  by \eqref{Jc0}, Claim~\ref{step4} and \eqref{sstu} we obtain
\begin{equation}\label{bo}
\mathcal J_{s,\,\lambda}(u_\infty)+\frac s n\|u_m-u_{\infty}\|^{2}_{X_0^s(\Omega)}=\mathcal J_{s,\,\lambda}(u_m) +\textit{o}(1)=c_2+\textit{o}(1).
\end{equation}
On the other hand, by \eqref{acot}, up to a subsequence, we can assume that
\begin{equation}\label{L1}
\|u_m-u_\infty\|_{X_0^s(\Omega)}^2\to L\geq0,
\end{equation}
and then, as a consequence of Claim~\ref{step5},
$$\int_\Omega |u_m(x)-u_\infty(x)|^{2^*_{s}}\,dx\to {\tilde{L}\geq}L.$$
By the definition of $S(n,s)$ given in \eqref{Ss}, we have
$$L\geq S(n,s){\tilde{L}}^{2/2^*_{s}}{\geq S(n,s)L^{2/2^*_{s}}}\,,$$
so that
$$L=0 \quad \mbox{or} \quad L\geq S(n,s)^{\frac{n}{2s}}\,.$$
We now prove that the case $L\geq S(n,s)^{\frac{n}{2s}}$ can not occur. Indeed taking $\varphi=u_\infty\in X_0^s(\Omega)$ as a test function in Claim~\ref{claim3}, we have that
$$\|u_{\infty}\|^{2}_{X_0^s(\Omega)} =\lambda \int_\Omega
{((u_\infty)_+)}^{q+1}\,dx+ \int_\Omega {((u_\infty)_+)}^{2^*_{s}}dx.$$
That is,
\begin{equation}\label{step3}
\mathcal J_{s,\,\lambda}(u_\infty)=\lambda\left(\frac 1 2 -\frac{1}{q+1}\right)\|{((u_\infty)_+)}\|_{L^{q+1}(\Omega)}^{q+1}+\frac s n\,\|{((u_\infty)_+)}\|_{L^{2^*_{s}}(\Omega)}^{2^*_{s}}\geq 0\,,
\end{equation}
thanks to the positivity of $\lambda$ and the fact that $q>1$\,.
Therefore, if $L\geq S(n,s)^{\frac{n}{2s}}$, then, by \eqref{bo}, \eqref{L1} and \eqref{step3} we get
$$c_2=\mathcal J_{s,\,\lambda}(u_\infty)+\frac s n\,L \geq \frac s n\,L \geq \frac s n\, S(n,s)^{\frac{n}{2s}}\,,$$
which contradicts the fact that $c_2<c^*$, for the $c^*$  given in~\eqref{cstar}\,. Thus $L=0$ and so, by \eqref{L1}, we obtain that
$$\|u_m-u_\infty\|_{X_0^s(\Omega)}\to 0.$$
\end{proof}
\begin{remark}
Note that the proof of Proposition~\ref{propPS} could be also obtained by the concentration-compactness theory of Subsection \ref{sec:compactness1}. This simply means that the arguments performed in the last part of the proof of Lemma~\ref{strongly_convergent} can be adapted to the convex setting.
\end{remark}

\subsection{Proof of Theorem~\ref{lapfra0}}\label{sec:dimtheorem1}
By Proposition \ref{MPG} and \eqref{limneg} we get that
$\mathcal J_{s, \,\lambda}$ satisfies the geometric features required by the MPT (see \cite{ar}). Moreover, by Proposition~\ref{propPS} the functional~$\mathcal J_{s,\,\lambda}$ verifies the Palais--Smale condition at any level $c$, provided $c<c^*$.

Now, as in the concave case, we find a path with energy below the critical level $c^*$. That is, we have the following.
\begin{lemma}\label{underlevel2}
Let $\lambda>0$, $c^*$ be as in \eqref{cstar} and $\eta_\varepsilon$ be the {non negative} function defined in \eqref{etaeps}.
Then, there exists $\varepsilon
>0$ small enough such that
$$\sup_{t\geq 0}\mathcal J_{s,\,\lambda}(t\eta_{\varepsilon}) <c^*\,,$$
provided
\begin{itemize}
\item $n>\frac{2s(q+3)}{q+1}$ and $\lambda>0$ or
\item $n\leq\frac{2s(q+3)}{q+1}$ and $\lambda>\lambda_s$, for a suitable $\lambda_s>0$.
\end{itemize}
\end{lemma}

\begin{proof}
Let $n>\frac{2s(q+3)}{q+1}$.

First of all note that since $q>1$ we get that $\displaystyle n>2s\left(1+\frac{1}{q}\right)$. Therefore, denoting by $N:=-(n-(n-2s)(q+1))>0$, for some positive constants $c$ and $\tilde{C}$, it follows that
\begin{eqnarray}
\int_{\mathbb{R}^{n}}{\eta_{\varepsilon}(x)^{q+1}\, dx}&=&C\int_{|x|<r}{u_{\varepsilon}^{q+1}\, dx}\nonumber\\
&=&C\varepsilon^{\left(\frac{n-2s}{2}\right)(q+1)}\int_{|x|<r}{\frac{dx}{(|x|^{2}+\varepsilon^2)^{\frac{(n-2s)(q+1)}{2}}}}\nonumber\\
&=&C\varepsilon^{-\left(\frac{n-2s}{2}\right)(q+1)}\int_{0}^{r}{\frac{\rho^{n-1}}{\left(1+\left(\frac{\rho}{\varepsilon}\right)^2\right)^{\frac{(n-2s)(q+1)}{2}}}\, d\rho}\nonumber\\
&=&C\varepsilon^{n-\left(\frac{n-2s}{2}\right)(q+1)}\int_{0}^{r/\varepsilon}{\frac{t^{n-1}}{(1+t^2)^{\frac{(n-2s)(q+1)}{2}}}\, dt}\nonumber\\
&\geq&C\varepsilon^{n-\left(\frac{n-2s}{2}\right)(q+1)}\int_{1}^{r/\varepsilon}{t^{n-1-(n-2s)(q+1)}\, dt}\nonumber\\
&=&\frac{C\varepsilon^{n-\left(\frac{n-2s}{2}\right)(q+1)}}{N}\left(1-\left(\frac{\varepsilon}{r}\right)^{N}\right)\nonumber\\
&\geq&\tilde{C}\varepsilon^{n-\left(\frac{n-2s}{2}\right)(q+1)}.\label{estimatesOld222}
\end{eqnarray}
Then, by \eqref{estimatesOld} and \eqref{estimatesOld222} for any $t\geq 0$ and $\varepsilon>0$ small enough we obtain
\begin{eqnarray}
\mathcal J_{s, \,\lambda}(t\eta_{\varepsilon})&=&\frac{t^{2}}{2}\|\eta_{\varepsilon}\|_{X^{s}_{0}(\Omega)}^{2}-\frac{t^{2^*_{s}}}{2^*_{s}}-\lambda\frac{t^{q+1}}{q+1}\int_{\Omega}{\eta_{\varepsilon}^{q+1}\, dx}\nonumber\\
&\leq&\frac{t^{2}}{2}(S(n,s)+C\varepsilon^{n-2s})-\frac{t^{2^*_{s}}}{2^*_{s}}-\tilde{C}\lambda\frac{t^{q+1}}{q+1}\varepsilon^{n-\left(\frac{n-2s}{2}\right)(q+1)}=:g(t).\label{Jg2}
\end{eqnarray}
It is clear that
\[\lim_{t\rightarrow\infty}g(t)=-\infty,\]
therefore $\sup_{t\geq0}{g(t)}$ is attained at
some $t_{\varepsilon,\lambda}:=t_{\varepsilon}\geq0$. As we comment in the proof of Lemma \ref{underlevel} we could suppose $t_{\varepsilon}>0$. Differentiating $g(t)$ and equaling to zero,
we obtain that
\begin{equation}\label{difJ2}
0=g'(t_{\varepsilon})=
t_{\varepsilon}(S(n,s)+C\varepsilon^{n-2s})-t_{\varepsilon}^{2^*_{s}
-1}-\tilde{C}\lambda t_{\varepsilon}^{q}\varepsilon^{n-\left(\frac{n-2s}{2}\right)(q+1)}.
\end{equation}
Hence,
$$t_{\varepsilon}<(S(n,s)+C\varepsilon^{n-2s})^{\frac{1}{2^*_{s}-2}}.$$
Moreover, we have that for $\varepsilon>0$ small enough
\begin{equation}\label{teps1}
t_\varepsilon\ge c>0.\end{equation}
Indeed, from \eqref{difJ2} it follows that
$$
t_\varepsilon^{2^*_{s}-2}+\tilde{C}\lambda t_{\varepsilon}^{q-1}\varepsilon^{n-\left(\frac{n-2s}{2}\right)(q+1)}=S(n,s)+C\varepsilon^{n-2s}\geq c>0,\, \mbox{for $\varepsilon>0$ small enough}.
$$
Also, since the function
$$
t \mapsto
\frac{t^2}{2}(S(n,s)+C\varepsilon^{n-2s})
- \frac{t^{2^*_{s}}}{2^*_{s}}
$$
is increasing on
$[0,(S(n,s)+C\varepsilon^{n-2s})^{\frac{1}{2^*_{s}-2}}]$,
by \eqref{Jg2} and \eqref{teps1} we obtain
\begin{eqnarray}
\displaystyle \sup_{t\geq 0}
g(t)=g(t_{\varepsilon})&\leq&
\frac{s}{n}(S(n,s)+C\varepsilon^{n-2s})^{\frac{n}{2s}}-
\overline{C}\varepsilon^{n-\left(\frac{n-2s}{2}\right)(q+1)}\nonumber\\
&\leq&\frac{s}{n}S(n,s)^{\frac{n}{2s}}+C\varepsilon^{n-2s}-
\overline{C}\varepsilon^{n-\left(\frac{n-2s}{2}\right)(q+1)},\label{sympli3}
\end{eqnarray}
for some $\overline{C}>0$.
Finally, from our hypothesis on $\displaystyle n$, 
we conclude from \eqref{sympli3}  that
$$\sup_{t\geq0}\mathcal J_{s,\,\lambda}(t\eta_{\varepsilon})\leq g(t_{\varepsilon})<\frac s n S(n,s)^{\frac{n}{2s}}.$$

\medskip

Consider now the case $n\leq \frac{2s(q+3)}{q+1}$.
Arguing exactly as in the previous case, 
we get that
\begin{equation}\label{sympli4}
(S(n,s)+C\varepsilon^{n-2s})=t_{\varepsilon,\lambda}^{2^*_{s}
-2}+\tilde{C}\lambda t_{\varepsilon,\lambda}^{q-1}\varepsilon^{n-\left(\frac{n-2s}{2}\right)(q+1)},\end{equation}
with $t_{\varepsilon,\lambda}>0$ the point where the $\sup_{t\geq0}g(t)$ is attained.
We claim that
\begin{equation}\label{zetalambda0}
t_{\varepsilon,\lambda} \to 0\quad\mbox{when}\quad\lambda \to +\infty.
\end{equation}
To see this assume that $\displaystyle\varlimsup_{\lambda\to \infty}t_{\varepsilon,\lambda}=\ell>0$. Then, passing to the limit when $\lambda \to +\infty$ in \eqref{sympli4} we would get $(S(n,s)+C\varepsilon^{n-2s})=+\infty$, which is a contradiction and \eqref{zetalambda0} follows.
If we take now $\beta$ the positive number given in Proposition \ref{MPG}, by \eqref{zetalambda0} we obtain that
\begin{eqnarray}
0\leq \sup_{t\ge 0}\mathcal J_{s,\,\lambda}(t \eta_\varepsilon)&\leq&g(t_{\varepsilon,\lambda})\nonumber\\
&=&\frac{t_{\varepsilon,\lambda}^{2}}{2}(S(n,s)+C\varepsilon^{n-2s})-\frac{t_{\varepsilon, \lambda}^{2^*_{s}}}{2^*_{s}}-\widetilde{C}\lambda\frac{t_{\varepsilon, \lambda}^{q+1}}{q+1}\varepsilon^{n-\left(\frac{n-2s}{2}\right)(q+1)}\nonumber\\
&\leq&\frac{t_{\varepsilon, \lambda}^{2}}{2}(S(n,s)+C\varepsilon^{n-2s})-\frac{t_{\varepsilon, \lambda}^{2^*_{s}}}{2^*_{s}}\to 0,\nonumber
\end{eqnarray}
when $\lambda\to\infty$. Then,
$$\lim_{\lambda\to +\infty}\sup_{t\geq 0} \mathcal J_{s,\,\lambda}(t \eta_\varepsilon)=0,$$
which easily yields the desired conclusion for the case $n\leq \frac{2s(q+3)}{q+1}$.
\end{proof}

We conclude now the proof of Theorem \ref{lapfra0}. In order to do so, we define
$$
\Gamma_\varepsilon=\{ \gamma\in\mathcal{C}([0,1],X_0^s(\Omega)) :\: \gamma(0)=0,\, \gamma(1)=M_\varepsilon \eta_\varepsilon\}
$$
for some $M_\varepsilon>0$ big enough such that $\mathcal{J}_{s,\,\lambda}(M_\varepsilon \eta_\varepsilon)<0$. Observe that
for every $\gamma\in \Gamma_{\varepsilon}$ the function $t\to\|\gamma(t)\|_{X_{0}^{s}(\Omega)}$ is continuous in $[0,1]$. Therefore, for the $\alpha$ given in Proposition \ref{MPG}, since $\|\gamma(0)\|_{X_{0}^{s}(\Omega)}=0<\alpha$ and $\|\gamma(1)\|_{X_{0}^{s}(\Omega)}=\|M_{\varepsilon}\eta_{\varepsilon}\|_{X_{0}^{s}(\Omega)}>\alpha$ for  $M_{\varepsilon}$ sufficiently large, there exists $t_0\in(0,1)$ such that $\|\gamma(t_0)\|_{X_{0}^{s}(\Omega)}=\alpha$. As a consequence,
$$\sup_{0\le t\le 1} \mathcal J_{s,\,\lambda}(\gamma(t))\geq \mathcal J_{s,\,\lambda}(\gamma(t_{0}))\geq\inf_{\|v\|_{X_{0}^{s}(\Omega)}=\alpha} \mathcal J_{s,\,\lambda}(v)\geq\beta>0,$$
where $\beta$ is the positive value given in Proposition \ref{MPG}. Hence,
$$c_\varepsilon=\inf_{\gamma\in \Gamma_\varepsilon}\sup_{0\le t\le 1} \mathcal J_{s,\,\lambda}(\gamma(t))>0.$$
Then, by Lemma \ref{underlevel2}, Proposition \ref{propPS} and the MPT given in \cite{ar} we conclude that the functional~$\mathcal J_{s,\,\lambda}$ admits a critical point $u\in X_0^s(\Omega)$, provided
$n>\frac{2s(q+3)}{q+1}$ and $\lambda>0$ or $n\leq\frac{2s(q+3)}{q+1}$ and $\lambda>\lambda_s$, for a suitable $\lambda_s>0$.
Moreover, since $\mathcal J_{s,\,\lambda}(u)=c_{\epsilon}\geq \beta>0$ and $\mathcal J_{s,\,\lambda}(0)=0$, the function $u$ is not the trivial one. This concludes the proof of Theorem~\ref{lapfra0}.

\begin{remark}
Some of the results obtained in Section~\ref{sec:q<1} and Section~\ref{sec:q>1} are true for integrodifferential operators more general than the fractional Laplacian, such as, for instance, the ones considered in \cite{svmountain, servadeivaldinociBN}.
\end{remark}

\end{document}